\documentclass[english]{amsart}
\usepackage[T1]{fontenc}
\usepackage[latin9]{inputenc}
\usepackage{geometry}
\usepackage{amsbsy}
\usepackage{amstext}
\usepackage{amsthm}
\usepackage{amssymb}
\usepackage{graphicx}
\usepackage{setspace}
\makeatletter
\numberwithin{equation}{section}
\numberwithin{figure}{section}
\theoremstyle{plain}
\newtheorem{thm}{\protect\theoremname}
\theoremstyle{definition}
\newtheorem{defn}[thm]{\protect\definitionname}
\theoremstyle{plain}
\newtheorem{lem}[thm]{\protect\lemmaname}
\theoremstyle{plain}
\newtheorem{prop}[thm]{\protect\propositionname}
\theoremstyle{remark}
\newtheorem{rem}[thm]{\protect\remarkname}
\theoremstyle{definition}
\newtheorem{example}[thm]{\protect\examplename}

\allowdisplaybreaks

\makeatother

\usepackage{babel}
\providecommand{\definitionname}{Definition}
\providecommand{\examplename}{Example}
\providecommand{\lemmaname}{Lemma}
\providecommand{\propositionname}{Proposition}
\providecommand{\remarkname}{Remark}
\providecommand{\theoremname}{Theorem}

\begin{document}
\title{Self-Exciting Multifractional Processes }
\author{Fabian A. Harang~~~~~~~ Marc Lagunas-Merino~~~~~~~ Salvador
Ortiz-Latorre \\
}
\date{\today}
\begin{abstract}
We propose a new multifractional stochastic process which allows for
self-exciting behavior, similar to what can be seen for example in
earthquakes and other self-organizing phenomena. The process can be
seen as an extension of a multifractional Brownian motion, where the
Hurst function is dependent on the past of the process. We define
this through a stochastic Volterra equation, and we prove existence
and uniqueness of this equation, as well as give bounds on the $p-$order
moments, for all $p\geq1$. We show convergence of an Euler-Maruyama
scheme for the process, and also give the rate of convergence, which
is depending on the self-exciting dynamics of the process. Moreover,
we discuss different applications of this process, and give examples
of different functions to model self-exciting behavior.
\end{abstract}

\maketitle

\section{Introduction and Notation }

In recent years, higher computer power and better tools from statistics
show that there are many natural phenomena which do not follow the
standard normal distribution, but rather exhibit different types of
memory, and sometimes changing these properties over time. Therefore
several different types of extensions of standard stochastic processes
have been proposed to try to give a more realistic picture of nature
corresponding to what we observe. There are several stochastic processes
which are popular today for the modeling of varying memory in a process,
one of them is known as the Hawkes process, see for example \cite{Ha18}.
This is a point process which allows for self-exciting behavior by
letting the conditional intensity to be dependent on the past events
of the process. In this note, we will consider a continuous type of
process which is inspired by the multifractional Brownian motion.
This process is interesting for being a non-stationary Gaussian process
which has regularity properties changing in time. A simple version
of this process is known as the Riemann-Liouville multifractional
Brownian motion and can be represented by the integral 
\begin{equation}
B_{t}^{h}=\int_{0}^{t}\left(t-s\right)^{h\left(t\right)-\frac{1}{2}}dB_{s},\label{eq:mBm}
\end{equation}
where $\left\{ B_{t}\right\} _{t\in[0,T]}$ is a Brownian motion and
$h$ is a deterministic function. Interestingly, if we restrict the
process to a small interval, say $\left[t-\epsilon,t+\epsilon\right]$,
the local $\alpha$-Hölder regularity of this process on that interval
is of order $\alpha\sim h(t)$ if $\epsilon$ is sufficiently small.
Thus the regularity of the process is depending on time. Applications
of such processes have been found in fields ranging from Internet
traffic and image synthesis to finance, see for example \cite{BeJaRo97,BiPaPi13,BiPaPi15,BiPi07,BiPi15,CoLeVe14,LeVe14,PiBiPa18}.
In 2010 D. Sornette and V. Filimonov proposed a self-excited multifractal
process to be considered in the modeling of earthquakes and financial
market crashes, see \cite{FiSo11}. By self-excited process, the authors
mean a process where the future state depends directly on all the
past states of the process. The model they proposed was defined in
a discrete manner. They also suggested a possible continuous time
version of their model, but they did not study its existence rigorously.
This article is therefore meant as an attempt to propose a continuous
time version of a similar model to that proposed by Sornette and Filimonov,
and we will study its mathematical properties. 

We will first consider an extension of a multifractional Brownian
process, which is found as the solution to the stochastic differential
equation 
\begin{equation}
X_{t}^{h}=\int_{0}^{t}\left(t-s\right)^{h(t,X_{s}^{h})-\frac{1}{2}}dB_{s},\label{eq:SEM}
\end{equation}
where $\left\{ B_{t}\right\} _{t\in[0,T]}$ is a general $d$-dimensional
Brownian motion, and $h$ is bounded and takes values in $(0,1)$.
Already at this point we could think that the local regularity of
the process $X$ would be depending on the history of $X$ through
$h$, in a similar manner as for the multifractional Brownian motion
in equation (\ref{eq:mBm}). As we can see, the formulation of the
process is through a stochastic Volterra equation with a possibly
singular kernel. We will therefore show the existence and uniqueness
of this equation, and then say that its solution is a Self-Exciting
Multifractional Process (SEM) $X^{h}$. We will study the probabilistic
properties, and discuss examples of functions $h$ which give different
dynamics for the process $X^{h}$. The process is neither stationary
nor Gaussian in general, and is therefore mathematically challenging
to apply in any standard model for example in finance but do, at this
point, have some interesting properties on its own. The study of such
processes could also shed some light on natural phenomena behaving
outside of the scope of standard stochastic processes, such as the
self-excited dynamics of earthquakes as they argue in \cite{FiSo11}. 

We will first show the existence and uniqueness of the Equation (\ref{eq:SEM})
and then study probabilistic and path properties such as variance
and regularity of the process. We will introduce an Euler-Maruyama
scheme to approximate the process, and show its strong convergence
as well as estimate its rate of convergence. Finally, we will discuss
an extension of the process to a Gamma type process, which might be
interesting for various applications. 

\subsection{Notation and preliminaries}

Let $T>0$ be a fixed constant. We will use the standard notation
$L^{\infty}\left(\left[0,T\right]\right)$ for essentially bounded
functions on the interval $\left[0,T\right]$. Furthermore, let $\triangle^{(m)}\left[a,b\right]$
denote the $m$-simplex. That is, define $\triangle^{(m)}[a,b]$ to
be given by 
\[
\triangle^{(m)}\left([a,b]\right):=\left\{ (s_{m},\ldots,s_{1}):a\leq s_{1}<\ldots<s_{m}\leq b\right\} .
\]
 We will consider functions $k:\triangle^{\left(2\right)}\left(\left[0,T\right]\right)\rightarrow\mathbb{R}_{+}$
which will be used as a kernel in an integral operator, in the sense
that we consider integrals of the form 
\[
\int_{0}^{t}k\left(t,s\right)f\left(s\right)ds,
\]
whenever the integral is well defined. We call these functions Volterra
kernels. 
\begin{defn}
Let $k:\triangle^{(2)}\left(\left[0,T\right]\right)\rightarrow\mathbb{R}_{+}$
be a Volterra kernel. If $k$ satisfies 
\[
t\mapsto\int_{0}^{t}k\left(t,s\right)ds\in L^{\infty}\left(\left[0,T\right]\right)
\]
and 
\[
\limsup_{\epsilon\downarrow0}\parallel\int_{\cdot}^{\cdot+\epsilon}k\left(\cdot+\epsilon,s\right)ds\parallel_{L^{\infty}\left(\left[0,T\right]\right)}=0,
\]
then we say that $k\in\mathcal{K}_{0}.$ 
\end{defn}

We will frequently use the constant $C$ to denote a general constant,
which might vary throughout the text. When it is important, we will
mention what this constant depends upon in subscript, i.e. $C=C_{T}$
to denote dependence in $T$ .

\section{Zhang's Existence and Uniqueness of Stochastic Volterra Equations}

In this section we will assume that $\left\{ B_{t}\right\} _{t\in\left[0,T\right]}$
is a $d$-dimensional Brownian motion defined on a filtered probability
space $(\Omega,\mathcal{F},\left\{ \mathcal{F}_{t}\right\} _{t\in\left[0,T\right]},P)$.
Consider the following Volterra equation

\begin{equation}
X_{t}=g\left(t\right)+\int_{0}^{t}\sigma\left(t,s,X_{s}\right)dB_{s},\qquad0\leq t\leq T,\label{eq:Stoch Volterra equation}
\end{equation}
where $g$ is a measurable, $\left\{ \mathcal{F}_{t}\right\} $-adapted
stochastic process and $\sigma:\triangle^{(2)}\left(\left[0,T\right]\right)\times\mathbb{R}^{n}\rightarrow\mathcal{L}\left(\mathbb{R}^{d},\mathbb{R}^{n}\right)$
is a measurable function, where $\mathcal{L}\left(\mathbb{R}^{d},\mathbb{R}^{n}\right)$
is the linear space of $d\times n$-matrices.

Next we write a simplified version of the hypotheses for $\sigma$
and $g$, introduced previously by Zhang in \cite{Zha10}, which will
be used to prove that there exists a unique solution to the equation
$\left(\ref{eq:Stoch Volterra equation}\right)$. 
\begin{description}
\item [{(H1)}] There exists $k_{1}\in\mathcal{K}_{0}$ such that the function
$\sigma$ satisfies the following linear growth inequality for all
$(s,t)\in\triangle^{(2)}\left([0,T]\right),$ and $x\in\mathbb{R}^{n}$,
\[
\left|\sigma\left(t,s,x\right)\right|^{2}\leq k_{1}\left(t,s\right)\left(1+\left|x\right|^{2}\right).
\]
\item [{(H2)}] There exists $k_{2}\in\mathcal{K}_{0}$ such that the function
$\sigma$ satisfies the following Lipschitz inequality for all $(s,t)\in\triangle^{(2)}\left([0,T]\right),$
$x,y\in\mathbb{R}^{n}$, 
\[
\left|\sigma\left(t,s,x\right)-\sigma\left(t,s,y\right)\right|^{2}\leq k_{2}\left(t,s\right)\left|x-y\right|^{2}.
\]
\item [{(H3)}] For some $p\geq2$, we have 
\[
\sup_{t\in\left[0,T\right]}\int_{0}^{t}\left[k_{1}\left(t,s\right)+k_{2}\left(t,s\right)\right]\cdot\mathbb{E}\left[\left|g\left(s\right)\right|^{p}\right]ds<\infty,
\]
where $k_{1}$ and $k_{2}$ satisfy $\mathbf{H1}$ and $\mathbf{H2}.$ 
\end{description}
Based on the above hypotheses, we can use the following tailor made
version of the theorem on existence and uniqueness found in \cite{Zha10}
to show that there exists a unique solution to equation $\left(\ref{eq:Stoch Volterra equation}\right)$.
\begin{thm}
\label{thm:Zhang Existence thm}$\left(\textnormal{Xicheng Zhang}\right)$
Assume that $\sigma:\triangle^{(2)}\left([0,T]\right)\times\mathbb{R}^{n}\rightarrow\mathcal{L}\left(\mathbb{R}^{d},\mathbb{R}^{n}\right)$
is measurable, and $g$ is an $\mathbb{R}^{n}$-valued, $\left\{ \mathcal{F}_{t}\right\} $-adapted
process satisfying $\mathbf{H1}-\mathbf{H3}$. Then there exists a
unique measurable, $\mathbb{R}^{n}$-valued, $\left\{ \mathcal{F}_{t}\right\} $-adapted
process $X_{t}$ satisfying for all $t\in[0,T]$ the equation 
\[
X_{t}=g\left(t\right)+\int_{0}^{t}\sigma\left(t,s,X_{s}\right)dB_{s}.
\]
 Furthermore, for some $C_{T,p,k_{1}}>0$ we have that 
\[
\mathbb{E}\left[\left|X_{t}\right|^{p}\right]\leq C_{T,p,k_{1}}\left(1+\mathbb{E}\left[\left|g\left(t\right)\right|^{p}\right]+\sup_{t\in[0,T]}\int_{0}^{t}k_{1}\left(t,s\right)\mathbb{E}\left[\left|g\left(s\right)\right|^{p}\right]ds\right),
\]
where $p$ is from $\mathbf{H3}.$
\end{thm}

It will also be useful, in future sections, to consider the following
additional hypothesis.
\begin{description}
\item [{(H4)}] The process $g$ is continuous and satisfies for some $\delta>0$
and for any $p\geq2$,
\[
\mathbb{E}\left[\sup_{t\in\left[0,T\right]}\left|g\left(t\right)\right|^{p}\right]<\infty,
\]
and 
\[
\mathbb{E}\left[\left|g\left(t\right)-g\left(s\right)\right|^{p}\right]\leq C_{T,p}\left|t-s\right|^{\delta p}.
\]
\end{description}

\section{Self-Exciting Multifractional Stochastic Processes}

Consider the stochastic process given formally by the Volterra equation
\[
X_{t}^{h}=g\left(t\right)+\int_{0}^{t}\left(t-s\right)^{h\left(t,X_{s}^{h}\right)-\frac{1}{2}}dB_{s},
\]
where $g$ is an $\left\{ \mathcal{F}_{t}\right\} $-adapted, one-dimensional
process, $h:\left[0,T\right]\times\mathbb{R}\rightarrow\mathbb{R}_{+}$
and $B$ is a one-dimensional Brownian motion. In this section, we
will show the existence and uniqueness of the solution for the above
equation by means of Theorem \ref{thm:Zhang Existence thm}. Moreover,
we will discuss the continuity properties of the solution. 
\begin{defn}
We say that a function $h:\left[0,T\right]\times\mathbb{R}\rightarrow\mathbb{R}_{+}$
is a Hurst function with parameters $(h_{*},h^{*})$, where $h_{*}\leq h^{*}$,
if $h\left(t,x\right)$ takes values in $\left[h_{*},h^{*}\right]\subset\left(0,1\right)$
for all $x\in\mathbb{R}^{d}$ and $t\in\left[0,T\right]$ and $h$
satisfies the following Lipschitz conditions for all $x,y\in\mathbb{R}$
and $t,t^{\prime}\in\left[0,T\right]$
\[
\left|h\left(t,x\right)-h\left(t,y\right)\right|\leq C\left|x-y\right|,
\]
\[
\left|h\left(t,x\right)-h\left(t^{\prime},x\right)\right|\leq C\left|t-t^{\prime}\right|,
\]
 for some $C>0$.
\end{defn}

\begin{lem}
\label{lem:Well defined sigma}Let $\sigma\left(t,s,x\right)=\left(t-s\right)^{h(t,x)-\frac{1}{2}}$
and let $h$ be a Hurst function with parameters $(h_{*},h^{*})$.
Then 
\begin{equation}
\left|\sigma\left(t,s,x\right)\right|^{2}\leq k\left(t,s\right)\left(1+\left|x\right|^{2}\right),\label{eq:sigma linear growth}
\end{equation}
 where 
\[
k\left(t,s\right)=C_{T}\left(t-s\right)^{2h_{*}-1},
\]
 and 
\begin{equation}
\left|\sigma\left(t,s,x\right)-\sigma\left(t,s,y\right)\right|^{2}\leq C_{T}k\left(t,s\right)\left|\log\left(t-s\right)\right|^{2}\left|x-y\right|^{2}.\label{eq:Lipschitz sigma}
\end{equation}
Moreover, $\sigma$ satisfies $\mathbf{H1}$-$\mathbf{H2}$. 
\end{lem}

\begin{proof}
We prove the three claims in the order they are stated in Lemma \ref{lem:Well defined sigma},
and start to prove equation $\left(\ref{eq:sigma linear growth}\right)$.
Remember that 
\[
h\left(t,x\right)\in\left[h_{*},h^{*}\right]\subset\left(0,1\right),
\]
 for all $t\in[0,T]$ and $x\in\mathbb{R}$, therefore we can trivially
find 
\begin{equation}
\left|\sigma\left(t,s,x\right)\right|^{2}=\left(t-s\right)^{2h(t,x)-1}=\left(t-s\right)^{2\left(h(t,x)-h_{*}\right)+2h_{*}-1}\leq T^{2\left(h^{*}-h_{*}\right)}\left(t-s\right)^{2h_{*}-1},\label{eq:BoundSigma2_No_x}
\end{equation}
which yields equation $\left(\ref{eq:sigma linear growth}\right)$
with $k\left(t,s\right)=C_{T}\left(t-s\right)^{2h_{*}-1}$. Next we
consider equation $\left(\ref{eq:Lipschitz sigma}\right)$, and using
that $x=\exp\left(\log\left(x\right)\right),$ we write 
\[
\sigma\left(t,s,x\right)=\exp\left(\left(\log\left(t-s\right)\right)\left(h(t,x)-\frac{1}{2}\right)\right),
\]
 where $(t,s)\in\triangle^{(2)}\left(\left[0,T\right]\right)$ and
consider the following inequality derived from the fundamental theorem
of calculus 
\begin{equation}
\left|e^{x}-e^{y}\right|\leq e^{\max(x,y)}\left|x-y\right|,\qquad x,y\in\mathbb{R}.\label{eq: expinequality}
\end{equation}
 Using the Lipschitz assumption on $h$ together with the above inequality,
we have that 
\begin{align*}
 & \left|\sigma\left(t,s,x\right)-\sigma\left(t,s,y\right)\right|^{2}\\
 & \leq\exp\left(2\max\left(\log\left(t-s\right)\left(h\left(t,x\right)-\frac{1}{2}\right),\log\left(t-s\right)\left(h\left(t,y\right)-\frac{1}{2}\right)\right)\right)\\
 & \quad\times\left|\log\left(t-s\right)\right|^{2}\left|h\left(t,x\right)-h\left(t,y\right)\right|^{2}\\
 & \leq C^{2}\exp\left(\max\left(\log\left(t-s\right)\left(2h\left(t,x\right)-1\right),\log\left(t-s\right)\left(2h\left(t,y\right)-1\right)\right)\right)\\
 & \quad\times\left|\log\left(t-s\right)\right|^{2}\left|x-y\right|^{2}.
\end{align*}
If $\left|t-s\right|\geq1$ then 
\[
\left|\sigma\left(t,s,x\right)-\sigma\left(t,s,y\right)\right|^{2}\leq C_{T}\left|x-y\right|^{2},
\]
 since $h$ is bounded. If $\left|t-s\right|<1$ then $\log\left(t-s\right)<0$
and, using that if $\theta<0$ then $\max\left(\theta x,\theta y\right)=\theta\min\left(x,y\right)$,
we have 
\begin{align*}
 & \left|\sigma\left(t,s,x\right)-\sigma\left(t,s,y\right)\right|^{2}\\
 & \leq C^{2}\exp\left(\log\left(t-s\right)\min\left(\left(2h\left(t,x\right)-1\right),\left(2h\left(t,y\right)-1\right)\right)\right)\left|\log\left(t-s\right)\right|^{2}\left|x-y\right|^{2}\\
 & \leq C^{2}\left|t-s\right|^{2h_{*}-1}\left|\log\left(t-s\right)\right|^{2}\left|x-y\right|^{2}.
\end{align*}
These estimates yield equation $\left(\ref{eq:Lipschitz sigma}\right)$.

Let $k$ be defined as above and $\nu\geq0$. Then, for any $0\leq a<T$,
$0\leq t<T-a$ and $\delta\in\left(0,1\right)$ fixed we have 
\begin{align*}
\phi_{\nu}\left(t,a\right) & :=\int_{a}^{a+t}k\left(a+t,s\right)\left|\log\left(a+t-s\right)\right|^{2\nu}ds\\
 & \leq C_{T}\int_{a}^{a+t}\left(a+t-s\right)^{\left(2h_{*}-1\right)}\left|\log\left(a+t-s\right)\right|^{2\nu}ds\\
 & =C_{T}\int_{0}^{t}u^{2h_{*}-1}\left|\log\left(u\right)\right|^{2\nu}du\\
 & \leq C_{T}C_{T,\delta}^{2\nu}\int_{0}^{t}u^{2h_{*}-1-2\nu\delta}du=C_{T,\delta,\nu,h_{*}}t^{2\left(h_{*}-\nu\delta\right)},
\end{align*}
where we have used that $\left|\log\left(u\right)\right|\leq C_{T,\delta}u^{-\delta}$
for some constant $C_{T,\delta}>0$. Note that $2\left(h_{*}-\nu\delta\right)>0$
if and only if $h_{*}>\nu\delta$. Therefore, choosing $a=0$, we
have that 
\[
t\mapsto\int_{0}^{t}k\left(t,s\right)\left|\log\left(t-s\right)\right|^{2\nu}ds\in L^{\infty}\left(\left[0,T\right]\right),
\]
if $h_{*}>0$, for $\nu=0$, and if $h_{*}>\delta$, for $\nu=1$.
Furthermore, by setting $a=t^{\prime}$ and $t=\epsilon$ in the estimate
for $\phi_{\nu}\left(t,a\right)$, we have
\begin{align*}
\limsup_{\epsilon\rightarrow0} & \parallel\int_{\cdot}^{\cdot+\epsilon}k\left(\cdot+\varepsilon,s\right)\left|\log\left(\cdot+\varepsilon-s\right)\right|^{2\nu}ds\parallel_{L^{\infty}\left(\left[0,T\right]\right)}\leq\limsup_{\epsilon\rightarrow0}C_{T,\delta,\nu,h_{*}}\varepsilon^{2\left(h_{*}-\nu\delta\right)}=0,
\end{align*}
if $h^{*}>0$, for $\nu=0$, and if $h_{*}>\delta$, for $\nu=1$.
Since $\delta$ can be chosen arbitrarily close to zero then $h_{*}$
can be arbitrarily close to zero. These estimates yield that $\sigma$
satisfies $\mathbf{H1\text{-}H3}$.
\end{proof}
Now, we can give the following theorem showing that the self-exciting
multifractional process from equation $\left(\ref{eq:SEM}\right)$
indeed exists and is unique. 
\begin{thm}
\label{thm:Existence and unique}Let $\sigma\left(t,s,x\right)=\left(t-s\right)^{h(t,x)-\frac{1}{2}}$
and $h$ be a Hurst function with parameters $\left(h_{*},h^{*}\right)$.
Moreover, let $g$ be an $\left\{ \mathcal{F}_{t}\right\} $-adapted,
$\mathbb{R}$-valued stochastic process satisfying $\mathbf{H3}.$
Then, there exists a unique process $X_{t}^{h}$ satisfying the equation
\begin{equation}
X_{t}^{h}=g\left(t\right)+\int_{0}^{t}\left(t-s\right)^{h\left(t,X_{t}^{h}\right)-\frac{1}{2}}dB_{s},\label{eq:SEMequation}
\end{equation}
 where $\left\{ B_{t}\right\} _{t\in\left[0,T\right]}$ is a one-dimensional
Brownian motion. Furthermore, we have the following inequality for
some $p\geq2$
\[
\mathbb{E}\left[\left|X_{t}^{h}\right|{}^{p}\right]\leq C_{T,p,k_{1}}\left(1+\mathbb{E}\left[\left|g\left(t\right)\right|{}^{p}\right]+\sup_{t\in\left[0,T\right]}\int_{0}^{t}\left(t-s\right)^{h_{*}-\frac{1}{2}}\mathbb{E}\left[\left|g\left(s\right)\right|{}^{p}\right]ds\right)
\]
We call this process a Self-Exciting Multifractional process (SEM)
.
\end{thm}

\begin{proof}
We have seen in Lemma \ref{lem:Well defined sigma} that $\sigma\left(t,s,x\right)=\left(t-s\right)^{h\left(t,x\right)-\frac{1}{2}}$
satisfies $\mathbf{H1-H2}.$ Applying Zhang's theorem gives the existence
and uniqueness and bounds on $p$-moments for the solution of (\ref{eq:SEMequation}). 
\end{proof}
Next we will show the Hölder regularity for the solution of (\ref{eq:SEMequation}).
We will need some preliminary lemmas.
\begin{lem}
\label{lem:FundIneq}Let $T>u>v>0.$ Then, for any $\alpha\leq0$
and $\beta\in\left[0,1\right]$ we have 
\[
\left|u^{\alpha}-v^{\alpha}\right|\leq2^{1-\beta}\left|\alpha\right|^{\beta}\left|u-v\right|^{\beta}\left|v\right|^{\alpha-\beta},
\]
and for $\alpha\in\left(0,1\right)$
\[
\left|u^{\alpha}-v^{\alpha}\right|\leq\left|\alpha\right|\left|u-v\right|^{\alpha+\beta\left(1-\alpha\right)}\left|v\right|^{-\beta\left(1-\alpha\right)}.
\]
\end{lem}

\begin{proof}
For $\alpha=0$ is clear. For $\alpha<1$ and $\alpha\neq0,$ using
the remainder of Taylor's formula in integral form we get
\begin{align}
\left|u^{\alpha}-v^{\alpha}\right| & =\left|\left(u-v\right)\int_{0}^{1}\alpha\left(v+\theta\left(u-v\right)\right)^{\alpha-1}\left(1-\theta\right)d\theta\right|\nonumber \\
 & \leq\left|\alpha\right|\left|u-v\right|\int_{0}^{1}\left|v+\theta\left(u-v\right)\right|^{\alpha-1}d\theta\leq\left|\alpha\right|\left|u-v\right|\left|v\right|^{\alpha-1},\label{eq: Bound1}
\end{align}
where we have used that $\left|v+\theta\left(u-v\right)\right|^{\alpha-1}\leq\left|v\right|^{\alpha-1}$.
Using that $\left|v+\theta\left(u-v\right)\right|^{\alpha-1}\leq\theta^{\alpha-1}\left|u-v\right|^{\alpha-1}$
and assuming that $\alpha\in\left(0,1\right)$ we obtain
\begin{equation}
\left|u^{\alpha}-v^{\alpha}\right|\leq\alpha\left|u-v\right|^{\alpha}\int_{0}^{1}\theta^{\alpha-1}d\theta=\left|u-v\right|^{\alpha}.\label{eq: Bound2}
\end{equation}
In what follows we will use the interpolation inequality $a\wedge b\leq a^{\beta}b^{1-\beta}$
for any $a,b>0$ and $\beta\in\left[0,1\right]$. 

Consider the case $\alpha<0.$ Using the interpolation inequality
with the simple bound $\left|u^{\alpha}-v^{\alpha}\right|\leq2\left|v\right|^{\alpha}$
and the bound (\ref{eq: Bound1}) we get
\begin{align*}
\left|u^{\alpha}-v^{\alpha}\right| & \leq2^{1-\beta}\left|\alpha\right|^{\beta}\left|u-v\right|^{\beta}\left|v\right|^{\beta\left(\alpha-1\right)}\left|v\right|^{(1-\beta)\alpha}=2^{1-\beta}\left|\alpha\right|^{\beta}\left|u-v\right|^{\beta}\left|v\right|^{\alpha-\beta}.
\end{align*}

Consider the case $\alpha\in\left(0,1\right)$. Using the interpolation
inequality with the bounds (\ref{eq: Bound1}) and (\ref{eq: Bound2})
we can write
\[
\left|u^{\alpha}-v^{\alpha}\right|\leq\left|\alpha\right|\left|u-v\right|^{\beta}\left|v\right|^{\beta\left(\alpha-1\right)}\left|u-v\right|^{\left(1-\beta\right)\alpha}=\left|\alpha\right|\left|u-v\right|^{\alpha+\beta\left(1-\alpha\right)}\left|v\right|^{-\beta\left(1-\alpha\right)}.
\]
\end{proof}
\begin{lem}
\label{lem:LambdaFunction}Let $\sigma\left(t,s,x\right)=\left(t-s\right)^{h(t,x)-\frac{1}{2}}$
and $h$ be a Hurst function with parameters $\left(h_{*},h^{*}\right)$.
Then, for any $0<\gamma<2h_{*}$ there exists $\lambda_{\gamma}:\triangle^{(3)}\left(\left[0,T\right]\right)\rightarrow\mathbb{R}$
such that 
\begin{equation}
\left|\sigma\left(t,s,x\right)-\sigma\left(t^{\prime},s,x\right)\right|^{2}\leq\lambda_{\gamma}\left(t,t^{\prime},s\right),\label{eq:time reg sigma}
\end{equation}
 and 
\begin{equation}
\int_{0}^{t^{\prime}}\lambda_{\gamma}\left(t,t^{\prime},s\right)ds\leq C_{T,\gamma}\left|t-t^{\prime}\right|^{\gamma},\label{eq:TimeLipschitzIntegralLambda}
\end{equation}
 for some constant $C_{T,\gamma}>0$.
\end{lem}

\begin{proof}
We have that
\[
\sigma\left(t,s,x\right)-\sigma\left(t^{\prime},s,x\right)=\left(t-s\right)^{h\left(t,x\right)-\frac{1}{2}}-\left(t^{\prime}-s\right)^{h\left(t^{\prime},x\right)-\frac{1}{2}}.
\]
Furthermore, notice that for all $t>t^{\prime}>s>0$, we can add and
subtract the term $\left(t-s\right)^{h(t^{\prime},x)-\frac{1}{2}}$
to get 
\[
\sigma\left(t,s,x\right)-\sigma\left(t^{\prime},s,x\right)=J^{1}\left(t,t^{\prime},s,x\right)+J^{2}\left(t,t^{\prime},s,x\right),
\]
where 
\begin{align*}
J^{1}\left(t,t^{\prime},s,x\right) & :=\left(t-s\right)^{h(t,x)-\frac{1}{2}}-\left(t-s\right)^{h(t^{\prime},x)-\frac{1}{2}},\\
J^{2}\left(t,t^{\prime},s,x\right) & :=\left(t-s\right)^{h(t^{\prime},x)-\frac{1}{2}}-\left(t^{\prime}-s\right)^{h(t^{\prime},x)-\frac{1}{2}}.
\end{align*}
 First we bound $J^{1}$. Using the inequality $\left(\ref{eq: expinequality}\right)$
and that $h$ is Lipschitz in the time argument, by similar arguments
as in Lemma \ref{lem:Well defined sigma}, we obtain for any $\delta\in\left(0,1\right)$
\begin{align*}
\left|J^{1}\left(t,t^{\prime},s,x\right)\right| & \leq C_{T}\left|t-t^{\prime}\right|\left|t-s\right|^{h_{*}-\frac{1}{2}}\left|\log\left(t-s\right)\right|\\
 & \leq C_{T,\delta}\left|t-t^{\prime}\right|\left|t-s\right|^{h_{*}-\frac{1}{2}-\delta}\\
 & \leq C_{T,\delta}\left|t-t^{\prime}\right|\left|t'-s\right|^{h_{*}-\frac{1}{2}-\delta},
\end{align*}
 since $s<t^{\prime}<t$. Next, in order to bound the term $J^{2}$,
we apply Lemma \ref{lem:FundIneq} with $u=t-s$,$v=t'-s$ and $\alpha=h(t^{\prime},x)-\frac{1}{2}$
. Note that, since $h(t^{\prime},x)\in\left[h_{*},h^{*}\right]\subset\left(0,1\right)$,
$\alpha\in\left[h_{*}-\frac{1}{2},h^{*}-\frac{1}{2}\right]\subset\left(-\frac{1}{2},\frac{1}{2}\right)$.
Hence, if $\alpha=$$h(t^{\prime},x)-\frac{1}{2}\leq0$ (this implies
$h_{*}<1/2$ and $\alpha\left(-\frac{1}{2},0\right)$), we get
\begin{align*}
\left|J^{2}\left(t,t^{\prime},s,x\right)\right| & \leq2\left|t-t^{\prime}\right|^{\beta_{1}}\left|t^{\prime}-s\right|^{h(t^{\prime},x)-\frac{1}{2}-\beta_{1}}\\
 & \leq C_{T}\left|t-t^{\prime}\right|^{\beta_{1}}\left|t^{\prime}-s\right|^{h_{*}-\frac{1}{2}-\beta_{1}},
\end{align*}
for any $\beta_{1}\in\left(0,1\right)$. If $\alpha=h(t^{\prime},x)-\frac{1}{2}>0$
(this implies $h^{*}>1/2$ and $\alpha\in\left(0,\frac{1}{2}\right)$),
we get 
\begin{align*}
\left|J^{2}\left(t,t^{\prime},s,x\right)\right| & \leq\frac{1}{2}\left|t-t^{\prime}\right|^{\alpha+\beta_{2}\left(1-\alpha\right)}\left|t^{\prime}-s\right|^{-\beta_{2}\left(1-\alpha\right)}\\
 & \leq\frac{1}{2}\left|t-t^{\prime}\right|^{\alpha+\frac{1}{2}-\varepsilon\left(1-\alpha\right)}\left|t^{\prime}-s\right|^{-\frac{1}{2}+\varepsilon\left(1-\alpha\right)}\\
 & \leq\frac{1}{2}\left|t-t^{\prime}\right|^{h_{*}-\varepsilon}\left|t^{\prime}-s\right|^{-\frac{1}{2}+\frac{\varepsilon}{2}},
\end{align*}
where in the second inequality we have chosen $\beta_{2}=\frac{1}{2\left(1-\alpha\right)}-\varepsilon,\varepsilon>0,$
and in the third inequality we have used that $(1-\alpha)\in\left(\frac{1}{2},1\right)$.
Therefore, we can write the following bound
\begin{align*}
\left|\sigma\left(t,s,x\right)-\sigma\left(t^{\prime},s,x\right)\right|^{2} & \leq2\left(\left|J^{1}\left(t,t^{\prime},s,x\right)\right|^{2}+\left|J^{2}\left(t,t^{\prime},s,x\right)\right|^{2}\right)\\
 & \leq2\left(C_{T,\delta}\left|t-t^{\prime}\right|^{2}\left|t'-s\right|^{2h_{*}-1-2\delta}\right.\\
 & \left.\quad+C_{T}\left|t-t^{\prime}\right|^{2\beta_{1}}\left|t^{\prime}-s\right|^{2h_{*}-1-2\beta_{1}}+\frac{1}{2}\left|t-t^{\prime}\right|^{2h_{*}-2\varepsilon}\left|t^{\prime}-s\right|^{-1+\varepsilon}\right)\\
 & \leq C_{T,\beta_{1}}\left|t-t^{\prime}\right|^{2\beta_{1}}\left|t'-s\right|^{-1+h_{*}-\beta_{1}},
\end{align*}
where to get the last inequality we have chosen $\delta=\beta_{1}$
and $\varepsilon=h_{*}-\beta_{1}$. Therefore, defining 
\[
\lambda_{\gamma}\left(t,t^{\prime},s\right):=C_{T,\gamma}\left(t-t^{\prime}\right)^{\gamma}\left(t^{\prime}-s\right)^{-1+h_{*}-\frac{\gamma}{2}},
\]
and choosing $\gamma$ such that $0<\gamma<2h_{*},$ we can compute
\[
\int_{0}^{t^{\prime}}\lambda_{\gamma}\left(t,t^{\prime},s\right)ds\leq C_{T,\gamma}\left(t^{\prime}\right)^{h_{*}-\frac{\gamma}{2}}\left(t-t^{\prime}\right)^{\gamma},
\]
 which concludes the proof.
\end{proof}
\begin{prop}
\label{prop:HolderContSEM}Let $\left\{ X_{t}^{h}\right\} _{t\in\left[0,T\right]}$
be a SEM process defined in Theorem \ref{thm:Existence and unique},
and assume that $g$ satisfies $\mathbf{H4}$ for some $\delta>0$.
Then there exists a set of paths $\mathcal{N}\subset\Omega$ with
$\mathbb{P}\left(\mathcal{N}\right)=0$, such that for all $\omega\in\mathcal{N}^{c}$
the path $X_{t}^{h}\left(\omega\right)$ has $\alpha$-Hölder continuous
trajectories for any $\alpha<h_{*}\wedge\delta$. In particular, we
have 
\[
\left|\left(X_{t}^{h}-X_{s}^{h}\right)\left(\omega\right)\right|\leq C\left(\omega\right)\left|t-s\right|^{\alpha},\qquad\omega\in\mathcal{N}^{c}.
\]
\end{prop}

\begin{proof}
By theorem \ref{thm:Existence and unique}, there exists a unique
solution $X_{t}^{h}$ to Equation (\ref{eq:SEMequation}) with bounded
$p$-order moments. We will show that $X_{t}^{h}$ also have Hölder
continuous paths. To this end, we will show that for any $p\in\mathbb{N}$
there exists a constant $C>0$ and a function $\alpha$, both depending
on $p$, such that 
\[
\mathbb{E}\left[\left|X_{s,t}^{h}\right|^{2p}\right]\leq C_{p}\left|t-s\right|^{\alpha_{p}},
\]
where $X_{s,t}^{h}=X_{t}^{h}-X_{s}^{h}$. From this we apply Kolmogorov's
continuity theorem (e.g. Theorem 2.8 in \cite{KarShre}, page 53)
in order to obtain the claim. Note that the increment of $X_{s,t}$
minus the increment of $g$ satisfies 
\[
X_{s,t}^{h}-\left(g\left(t\right)-g\left(s\right)\right)=\int_{s}^{t}\left(t-r\right)^{h\left(t,X_{r}\right)-\frac{1}{2}}dB_{r}+\int_{0}^{s}\left(t-r\right)^{h\left(t,X_{r}\right)-\frac{1}{2}}-\left(s-r\right)^{h\left(t,X_{r}\right)-\frac{1}{2}}dB_{r},
\]
and thus using that 
\begin{equation}
\left|a+b\right|^{q}\leq2^{q-1}\left(\left|a\right|^{q}+\left|b\right|^{q}\right),\label{eq:simpleQineq}
\end{equation}
 for any $q\in\mathbb{N},$ we obtain 
\begin{align*}
\mathbb{E}\left[\left|X_{s,t}^{h}-\left(g\left(t\right)-g\left(s\right)\right)\right|^{2p}\right] & \leq C_{p}\mathbb{E}\left[\left|\int_{s}^{t}\left(t-r\right)^{h\left(t,X_{r}\right)-\frac{1}{2}}dB_{r}\right|^{2p}\right]\\
 & +C_{p}\mathbb{E}\left[\left|\int_{0}^{s}\left(t-r\right)^{h\left(t,X_{r}\right)-\frac{1}{2}}-\left(s-r\right)^{h\left(t,X_{r}\right)-\frac{1}{2}}dB_{r}\right|^{2}\right]\\
 & =:C_{p}\left(J_{s,t}^{1}+J_{s,t}^{2}\right).
\end{align*}
Clearly, as $h\left(t,x\right)\in\left[h_{*},h^{*}\right]\subset\left(0,1\right)$,
we have by the Burkholder-Davis-Gundy (BDG) inequality that 
\begin{align}
J_{s,t}^{1} & \leq C_{p}\mathbb{E}\left[\left|\int_{s}^{t}\left(t-r\right)^{2h\left(t,X_{r}\right)-1}dr\right|^{p}\right]\label{eq: J1}\\
 & \leq C_{p,T}\left|\int_{s}^{t}\left(t-r\right)^{2h_{*}-1}dr\right|^{p}=C_{p,T,h_{*}}\left|t-s\right|^{2ph_{*}}.\nonumber 
\end{align}
 Consider now the term $J_{s,t}^{2}$. Applying again BDG inequality
together with the bounds $\left(\ref{eq:time reg sigma}\right)$ and
$\left(\ref{eq:TimeLipschitzIntegralLambda}\right)$ from Lemma \ref{lem:LambdaFunction},
we have that for any $\gamma<2h_{*}$
\begin{align}
J_{s,t}^{2} & \leq C_{p}\mathbb{E}\left[\left|\int_{0}^{s}\left[\left(\left(t-r\right)^{h\left(t,X_{r}\right)-\frac{1}{2}}-\left(s-r\right)^{h\left(t,X_{r}\right)-\frac{1}{2}}\right)^{2}\right]dr\right|^{p}\right]\nonumber \\
 & \leq C_{p}\mathbb{E}\left[\left|\int_{0}^{s}\lambda_{\gamma}\left(t,s,r\right)dr\right|^{p}\right]\leq C_{p,T,\gamma}\left|t-s\right|^{p\gamma},\label{eq: J2}
\end{align}
 Combining (\ref{eq: J1}) and (\ref{eq: J2}) we can see that 
\[
\mathbb{E}\left[\left|X_{s,t}^{h}-\left(g\left(t\right)-g\left(s\right)\right)\right|^{2p}\right]\leq C_{p,T,\gamma}\left|t-s\right|^{p\gamma}.
\]
Furthermore, again using (\ref{eq:simpleQineq}) we see that 
\[
\mathbb{E}\left[\left|X_{s,t}\right|^{2p}\right]\leq2^{2p-1}\left(\mathbb{E}\left[\left|X_{s,t}^{h}-\left(g\left(t\right)-g\left(s\right)\right)\right|^{2p}\right]+\mathbb{E}\left[\left|\left(g\left(t\right)-g\left(s\right)\right)\right|^{2p}\right]\right).
\]
Thus invoking the bounds from $\left({\bf H4}\right)$ on $g$, we
obtain that 
\[
\mathbb{E}\left[\left|X_{s,t}\right|^{2p}\right]\leq C_{p,T,\gamma}\left|t-s\right|^{2p\left(\frac{\gamma}{2}\wedge\delta\right)},
\]
and it follows from Kolmogorov's continuity theorem that $X^{h}$
has $\mathbb{P}$-a.s. $\alpha$-Hölder continuous trajectories with
$\alpha\in\left(0,h_{*}\wedge\delta\right)$. 
\end{proof}

\section{Simulation of Self-Exciting Multifractional Stochastic Processes}

The aim of this section is to study a discretization scheme for self-excited
multifractional (SEM) processes proposed in the previous sections.
In particular we will consider an Euler type discretization and prove
that converges strongly to the original process at a rate depending
on $h_{*}$. We end the section providing two examples of numerical
simulations using the Euler discretization. 

\subsection{Euler-Maruyama Approximation Scheme}

Consider a time discretization of the interval $\left[0,T\right],$
using a step-size $\Delta t=\frac{T}{N}>0$. The discrete time Euler-Maruyama
scheme (EM) is given by 
\begin{align}
\bar{X}_{0}^{h} & =X_{0}^{h}=0\\
\bar{X}_{k}^{h} & =\sum_{i=0}^{k-1}\left(t_{k}-t_{i}\right)^{h\left(t_{k},\bar{X}_{i}^{h}\right)-\frac{1}{2}}\Delta B_{i},\qquad k\in\left\{ 1,\ldots,N\right\} ,\label{eq:EM_1}
\end{align}
where $\Delta B_{i}=B\left(t_{i+1}\right)-B\left(t_{i}\right)$, and
yields an approximation of $X_{t_{k}}^{h}$ for $t_{k}=k\Delta t$
with $k\in\left\{ 0,\ldots,N\right\} .$ 

In order to study the approximation error, it is convenient to consider
the continuous time interpolation of $\left\{ \bar{X}_{k}^{h}\right\} _{k\in\left\{ 0,\ldots,N\right\} }$
given by 
\begin{equation}
\bar{X}_{t}^{h}=\int_{0}^{t}\left(t-\eta\left(s\right)\right)^{h\left(t,\bar{X}_{\eta\left(s\right)}^{h}\right)-\frac{1}{2}}dB_{s},\qquad t\in\left[0,T\right],\label{eq:DefContEuler}
\end{equation}
where $\eta\left(s\right):=t_{i}\cdot\boldsymbol{1}_{\left[t_{i},t_{i+1}\right)}\left(s\right)$. 

The following theorem is the main result in this section and its proof
uses Lemmas \ref{lem:Bound=000026TimeLipschContEuler} and \ref{theo:VolterraGronwall},
see below.
\begin{thm}
\label{thm:EM_StrongConvergence} Let $h$ be a Hurst function with
parameters $\left(h_{*},h^{*}\right)$ and let $X_{t}^{h}$ be the
solution of equation $\left(\ref{eq:SEM}\right)$. Then the Euler-Maruyama
scheme $\left(\ref{eq:DefContEuler}\right)$, satisfies 
\begin{equation}
\sup_{0\leq t\leq T}\mathbb{E}\left[\left|X_{t}^{h}-\bar{X}_{t}^{h}\right|^{2}\right]\leq C_{T,\gamma,h_{*}}E_{h_{*}}\left(C_{T,\gamma,h_{*}}\Gamma\left(h_{*}\right)T^{h_{*}}\right)\left|\Delta t\right|^{\gamma},\label{eq:EM_2}
\end{equation}
 where $\gamma\in\left(0,2h_{*}\right)$, and $C_{T,\gamma,h_{*}}$
is a positive constant, which does not depend on $N$.
\end{thm}

\begin{proof}
Define 
\[
\delta_{t}:=X_{t}^{h}-\bar{X}_{t}^{h},\qquad\varphi\left(t\right):=\sup_{0\leq s\leq t}\mathbb{E}\left[\left|\delta_{s}\right|^{2}\right],\quad t\in\left[0,T\right].
\]
For any $t\in\left[0,T\right],$ we can write
\begin{align*}
\delta_{t} & =\int_{0}^{t}\left(\left(t-s\right)^{h\left(t,X_{s}^{h}\right)-\frac{1}{2}}-\left(t-\eta\left(s\right)\right)^{h\left(t,\bar{X}_{\eta\left(s\right)}^{h}\right)-\frac{1}{2}}\right)dB_{s}\\
 & =\int_{0}^{t}\left(\left(t-s\right)^{h\left(t,X_{s}^{h}\right)-\frac{1}{2}}-\left(t-s\right)^{h\left(t,\bar{X}_{\eta\left(s\right)}^{h}\right)-\frac{1}{2}}\right)dB_{s}\\
 & \quad+\int_{0}^{t}\left(\left(t-s\right)^{h\left(t,\bar{X}_{\eta\left(s\right)}^{h}\right)-\frac{1}{2}}-\left(t-\eta\left(s\right)\right)^{h\left(t,\bar{X}_{\eta\left(s\right)}^{h}\right)-\frac{1}{2}}\right)dB_{s}\\
 & =:I_{1}\left(t\right)+I_{2}\left(t\right).
\end{align*}
First we bound the second moment of $I_{1}\left(t\right)$ in terms
of a Volterra integral of $\varphi$. Using the Itô isometry, equation
$\left(\ref{eq:Lipschitz sigma}\right)$ and the Lipschitz property
of $h$ we get 
\begin{align*}
\mathbb{E}\left[\left|I_{1}\left(t\right)\right|^{2}\right] & \leq\int_{0}^{t}k\left(t,s\right)\left(\log\left(t-s\right)\right)^{2}\mathbb{E}\left[\left(h\left(t,X_{s}^{h}\right)-h\left(t,\bar{X}_{\eta\left(s\right)}^{h}\right)\right)^{2}\right]ds\\
 & \leq C_{T,\delta}\int_{0}^{t}\left(t-s\right)^{2\left(h_{*}-\delta\right)-1}\mathbb{E}\left[\left|X_{s}^{h}-\bar{X}_{\eta\left(s\right)}^{h}\right|^{2}\right]ds,
\end{align*}
for $\delta>0$, arbitrarily small. By adding and subtracting $\bar{X}_{s}^{h}$,
we easily get that
\[
\mathbb{E}\left[\left|X_{s}^{h}-\bar{X}_{\eta\left(s\right)}^{h}\right|^{2}\right]\leq2\varphi\left(s\right)+2\mathbb{E}\left[\left|\bar{X}_{s}^{h}-\bar{X}_{\eta\left(s\right)}^{h}\right|^{2}\right],
\]
 Moreover, combining equation $\left(\ref{eq:SecondMomentTLCE}\right)$
in Lemma \ref{lem:Bound=000026TimeLipschContEuler}, yields
\[
\int_{0}^{t}\left(t-s\right)^{2\left(h_{*}-\delta\right)-1}\mathbb{E}\left[\left|\bar{X}_{s}^{h}-\bar{X}_{\eta\left(s\right)}^{h}\right|^{2}\right]ds\leq C_{T}\frac{T^{2\left(h_{*}-\delta\right)}}{2\left(h_{*}-\delta\right)}\left|\Delta t\right|^{\gamma}.
\]
Therefore, choosing $\delta=\frac{h_{*}}{2}$
\begin{equation}
\mathbb{E}\left[\left|I_{1}\right|^{2}\right]\leq C_{T,h_{*}}\left\{ \int_{0}^{t}\left(t-s\right)^{h_{*}-1}\varphi\left(s\right)ds+\left|\Delta t\right|^{\gamma}\right\} .\label{eq:I1Bound}
\end{equation}
Next, we find a bound for the second moment of $I_{2}\left(t\right)$.
Using again the Itô isometry, equations $\left(\ref{eq:time reg sigma}\right)$
and $\left(\ref{eq:TimeLipschitzIntegralLambda}\right),$ and Lemma
\ref{lem:Bound=000026TimeLipschContEuler} we can write 
\begin{align}
\mathbb{E}\left[\left|I_{2}\left(t\right)\right|^{2}\right] & \leq\int_{0}^{t}\lambda_{\gamma}\left(t+\left(s-\eta\left(s\right)\right),t,s\right)ds\leq C_{T,\gamma}\left|\Delta t\right|^{\gamma},\label{eq:I2Bound}
\end{align}
for any $\gamma<2h_{*}$. Combining the inequalities $\left(\ref{eq:I1Bound}\right)$
and $\left(\ref{eq:I2Bound}\right)$ we obtain
\[
\varphi\left(t\right)\leq C_{T,\gamma,h_{*}}\left\{ \int_{0}^{t}\left(t-s\right)^{h_{*}-1}\varphi\left(s\right)ds+\left|\Delta t\right|^{\gamma}\right\} .
\]
Using Theorem \ref{theo:VolterraGronwall} with $a\left(t\right)=C_{T,\gamma,h_{*}}\left|\Delta t\right|^{\gamma},$$g\left(t\right)=C_{T,\gamma,h_{*}}$
and $\beta=h_{*}$ we can conclude that 
\[
\varphi\left(T\right)\leq C_{T,\gamma,h_{*}}E_{h_{*}}\left(C_{T,\gamma,h_{*}}\Gamma\left(h_{*}\right)T^{h_{*}}\right)\left|\Delta t\right|^{\gamma}.
\]
\end{proof}
\begin{rem}
In \cite{Zha08}, Zhang introduced an Euler type scheme for stochastic
differential equations of Volterra type and showed that his scheme
converges at a certain positive rate, without being very precise.
A direct application of his result to our case provides a worse rate
than the one we obtain in Theorem \ref{thm:EM_StrongConvergence}.
The reason being that, due to our particular kernel, we are able to
use a fractional Gronwall lemma.
\end{rem}

\begin{lem}
\label{lem:Bound=000026TimeLipschContEuler} Let $h$ be a Hurst function
with parameters $\left(h_{*},h^{*}\right)$ and let $\bar{X}^{h}=\left\{ \bar{X}_{t}^{h}\right\} _{t\in\left[0,T\right]}$
be given by $\left(\ref{eq:DefContEuler}\right)$. Then
\begin{equation}
\mathbb{E}\left[\left|\bar{X}_{t}^{h}\right|^{2}\right]\leq C_{T},\quad0\leq t\leq T,\label{eq:SecondMomentContEuler}
\end{equation}
and 
\begin{equation}
\mathbb{E}\left[\left|\bar{X}_{t}^{h}-\bar{X}_{t^{\prime}}^{h}\right|^{2}\right]\leq C_{T,\gamma}\left|t-t^{\prime}\right|^{\gamma},\quad0\leq t^{\prime}\leq t\leq T,\label{eq:SecondMomentTLCE}
\end{equation}
 for any $\gamma<2h_{*}$, where $C_{T}$ and $C_{T,\gamma}$ are
positive constants.
\end{lem}

\begin{proof}
Recall that $k\left(t,s\right)=C_{T}\left(t-s\right)^{2h_{*}-1}$
and, since $\eta\left(s\right)\leq s$, we have the following inequality

\begin{equation}
k\left(t,\eta\left(s\right)\right)\leq k\left(t,s\right).\label{eq: Volterrabound}
\end{equation}
Using the Itô isometry, equation $\left(\ref{eq:BoundSigma2_No_x}\right)$
and equation (\ref{eq: Volterrabound}), we obtain
\begin{align*}
\mathbb{E}\left[\left|\bar{X}_{t}^{h}\right|^{2}\right] & =\mathbb{E}\left[\int_{0}^{t}\left(t-\eta\left(s\right)\right)^{2h\left(t,\bar{X}_{\eta\left(s\right)}^{h}\right)-1}ds\right]\\
 & \leq\int_{0}^{t}k\left(t,\eta\left(s\right)\right)ds\leq\int_{0}^{t}k\left(t,s\right)ds\leq C_{T}.
\end{align*}
To prove the bound $\left(\ref{eq:SecondMomentTLCE}\right),$ note
that
\begin{align*}
\bar{X}_{t}^{h}-\bar{X}_{t^{\prime}}^{h} & =\int_{t^{\prime}}^{t}\left(t-\eta\left(s\right)\right)^{h\left(t,\bar{X}_{\eta\left(s\right)}^{h}\right)-\frac{1}{2}}dB_{s},\\
 & \quad+\int_{0}^{t^{\prime}}\left\{ \left(t-\eta\left(s\right)\right)^{h\left(t,\bar{X}_{\eta\left(s\right)}^{h}\right)-\frac{1}{2}}-\left(t^{\prime}-\eta\left(s\right)\right)^{h\left(t^{\prime},\bar{X}_{\eta\left(s\right)}^{h}\right)-\frac{1}{2}}\right\} dB_{s}\\
 & =:J_{1}+J_{2}.
\end{align*}
 Due to the Itô isometry, equation $\left(\ref{eq:BoundSigma2_No_x}\right)$
and $\left(\ref{eq: Volterrabound}\right)$, we obtain the bounds

\begin{align*}
\mathbb{E}\left[\left|J_{1}\right|^{2}\right] & =\mathbb{E}\left[\int_{t^{\prime}}^{t}\left(t-\eta\left(s\right)\right)^{2h\left(t,\bar{X}_{\eta\left(s\right)}^{h}\right)-1}ds\right]\\
 & \leq\int_{t'}^{t}k\left(t,\eta\left(s\right)\right)ds\leq\int_{t'}^{t}k\left(t,s\right)ds=C_{T}\left|t-t^{\prime}\right|^{2h_{*}}.
\end{align*}
Using again the Itô isometry, equation $\left(\ref{eq:time reg sigma}\right)$
and equation $\left(\ref{eq:TimeLipschitzIntegralLambda}\right)$
we can write, for any $\gamma<2h_{*},$ that
\begin{align*}
\mathbb{E}\left[\left|J_{2}\right|^{2}\right] & \leq\int_{0}^{t^{\prime}}\lambda_{\gamma}\left(t,t^{\prime},\eta\left(s\right)\right)ds\leq\int_{0}^{t^{\prime}}\lambda_{\gamma}\left(t,t^{\prime},s\right)ds\leq C_{T,\gamma}\left|t-t^{\prime}\right|^{\gamma},
\end{align*}
where in the second inequality we have used $\lambda_{\gamma}\left(t,t^{\prime},\eta\left(s\right)\right)\leq\lambda_{\gamma}\left(t,t^{\prime},s\right)$,
because $\lambda_{\gamma}$ is essentially a negative fractional power
of $(t-s)$ and $\eta\left(s\right)\leq s$. Combining the bounds
for $\mathbb{E}\left[\left|J_{1}\right|^{2}\right]$ and $\mathbb{E}\left[\left|J_{2}\right|^{2}\right]$
the result follows.
\end{proof}
The following result is a combination of Theorem 1 and Corollary 2
in \cite{YeGaoDing07}.
\begin{thm}
\label{theo:VolterraGronwall}Suppose $\beta>0,$$a\left(t\right)$
is a nonnegative function locally integrable on $0\leq t<T<+\infty$
and $g\left(t\right)$ is a nonnegative, nondecreasing continuous
function defined on $0\leq t<T$, $g\left(t\right)\leq M$ (constant),
and suppose $u\left(t\right)$ is nonnegative and locally integrable
on $0\leq t<T$ with 
\[
u\left(t\right)\leq a\left(t\right)+g\left(t\right)\int_{0}^{t}\left(t-s\right)^{\beta-1}u\left(s\right)ds,
\]
on this interval. Then,
\[
u\left(t\right)\leq a\left(t\right)+\int_{0}^{t}\left(\sum_{n=1}^{\infty}\frac{\left(g\left(t\right)\Gamma\left(\beta\right)\right)^{n}}{\Gamma\left(n\beta\right)}\left(t-s\right)^{n\beta-1}a\left(s\right)\right)ds,\qquad0\leq t<T.
\]
If in addition, $a\left(t\right)$ is a nondecreasing function on
$\left[0,T\right)$. Then,
\[
u\left(t\right)\leq a\left(t\right)E_{\beta}\left(g\left(t\right)\Gamma\left(\beta\right)t^{\beta}\right),
\]
where $E_{\beta}$ is the Mittag-Leffler function defined by $E_{\beta}\left(z\right)=\sum_{k=0}^{\infty}\frac{z^{k}}{\Gamma\left(k\beta+1\right)}$.
\end{thm}

\subsection{Examples}

Let us now discuss some functions $h:\mathbb{R}\rightarrow\left(0,1\right)$
which produce some interesting self-exciting processes. 
\begin{example}
\label{exa:1}Let $h\left(x\right)=\frac{1}{2}+\frac{1/2}{1+x^{2}}\in\left(\frac{1}{2},1\right)\subset\left(0,1\right),$
and $\left\{ B_{t}\right\} _{t\in[0,T]}$ be a one-dimensional Brownian
motion. Assume $X_{t}^{h}$ starts at zero and define the process
as given in equation $\left(\ref{eq:SEM}\right)$. Figure $\left(\ref{plot: SEM_ProcessEX1}\right)$
shows the plot of $h$ on the left hand side and a sample path of
the process, on the right hand side, resulting from the implementation
\footnote{All simulations were run with a step-size $\Delta t=1/100$.}
of the EM-approximation given by equation $\left(\ref{eq:EM_1}\right)$.
Notice the fact that this process is smoother than a Brownian motion
at the origin and rapidly converges to the classical Brownian motion
as as the process departs from zero. This implies that $h\rightarrow\frac{1}{2}$
having only $h=1$ any time the sample path crossed the $x$-axis
again.

\begin{figure}
\centering{}\includegraphics[width=0.49\textwidth]{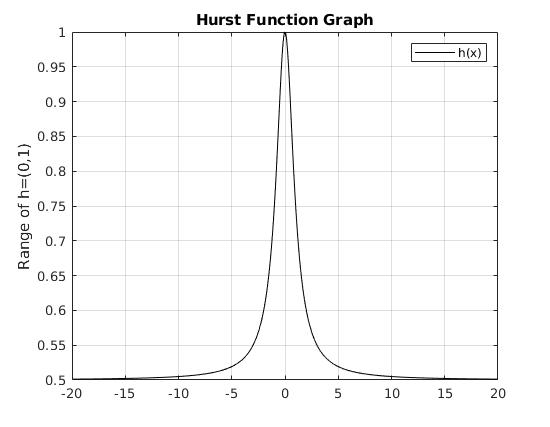}\includegraphics[width=0.49\textwidth]{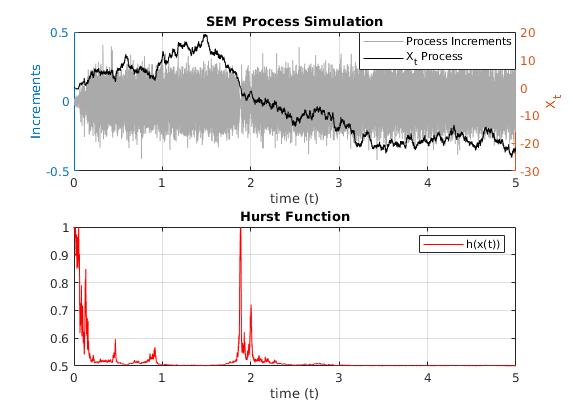}\caption[Simulation of a SEM Process I.]{Numerical simulation of a trajectory of a SEM Process given the Hurst
function is $h\left(x\right)=\frac{1}{2}+\frac{1/2}{1+x^{2}}$.}
\label{plot: SEM_ProcessEX1}
\end{figure}
\label{exa:2}Let $h\left(x\right)=\frac{1}{2}-\frac{1/2}{1+x^{2}}\in\left(0,\frac{1}{2}\right)\subset\left(0,1\right),$
and $\left\{ B_{t}\right\} _{t\in[0,T]}$ be a one-dimensional Brownian
motion. Assume $X_{t}^{h}$ starts at zero and define the process
as given in equation $\left(\ref{eq:SEM}\right)$. Figure $\left(\ref{plot: SEM_ProcessEX2}\right)$
shows the plot of $h$ on the left hand side and a sample path of
the process, on the right hand side, resulting from the implementation
of the EM-approximation given by equation $\left(\ref{eq:EM_1}\right)$.
Is interesting noticing in this case, contrary to the previous example,
that we have a rougher process than a Brownian motion at the origin,
temporarily resembles the classical Brownian motion as the sample
path departs from zero and gets rougher again whenever the process
crosses the $x$-axis. This makes the process go away from zero even
faster due to the increased roughness.

\begin{figure}
\centering{}\includegraphics[width=0.49\textwidth]{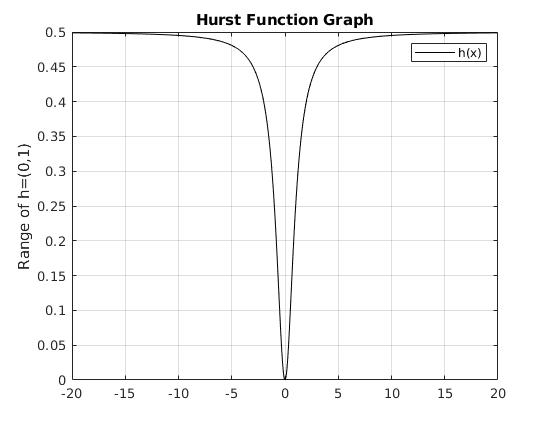}\includegraphics[width=0.49\textwidth]{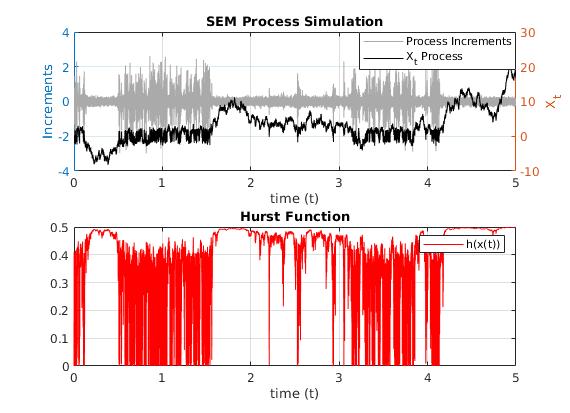}\caption[Simulation of a SEM Process II.]{Numerical simulation of a trajectory of a SEM Process given the Hurst
function is $h\left(x\right)=\frac{1}{2}-\frac{1/2}{1+x^{2}}$.}
\label{plot: SEM_ProcessEX2}
\end{figure}
\label{exa:3}Let $h\left(x\right)=\frac{1}{1+x^{2}}\in\left(0,1\right),$
and $\left\{ B_{t}\right\} _{t\in[0,T]}$ be a one-dimensional Brownian
motion. Assume $X_{t}^{h}$ starts at zero and define the process
as given in equation $\left(\ref{eq:SEM}\right)$. Figure $\left(\ref{plot: SEM_ProcesEX3}\right)$
shows the plot of $h$ on the left hand side and a sample path of
the process, on the right hand side, resulting from the implementation
\footnote{All simulations were run with a step-size $\Delta t=1/100$.}
of the EM-approximation given by equation $\left(\ref{eq:EM_1}\right)$.
Notice the fact that the Hurst function collapses to zero as the process
departs from zero, making the process be the roughest possible. Therefore
we would only recover smoother values, in particular $h=1$ only the
time the sample path crossed the $x$-axis again.

\begin{figure}
\begin{centering}
\includegraphics[width=0.49\textwidth]{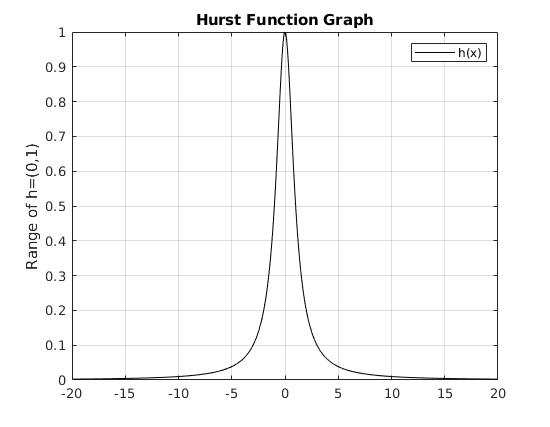}\includegraphics[width=0.49\textwidth]{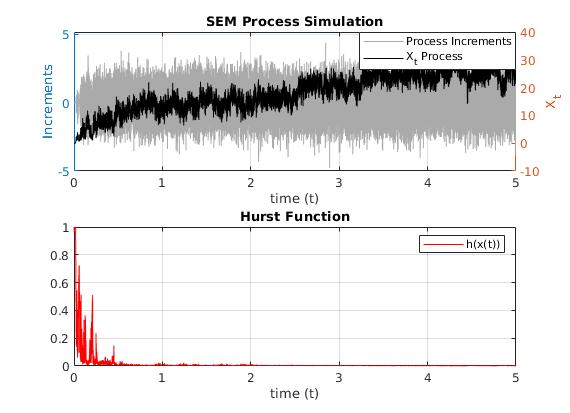}
\par\end{centering}
\caption[Simulation of a SEM Process III.]{Numerical simulation of a trajectory of a SEM Process given the Hurst
function is $h\left(x\right)=\frac{1}{1+x^{2}}$.}

\label{plot: SEM_ProcesEX3}
\end{figure}
\end{example}

\section{Self-Exciting Multifractional Gamma Processes}

Barndorff-Nielsen in \cite{B-N16}, introduces a class of self-exciting
gamma type of process, in order to model turbulence, because it captures
the intermittency effect observed in turbulent data. We would also
like to extend our process in order to capture the previously mentioned
intermittency effect. One could believe that if we were to choose
a trigonometric function as a Hurst function, i.e. $h\left(t,x\right)=\alpha+\beta\cdot\sin\left(\gamma x\right),$
for some $\alpha,\beta,\gamma\in\mathbb{R}$ in the SEM process, then
we might observe a regime switch in the Hurst parameters. Since the
values of the process $X_{t}$ may get very large, the oscillations
may take place more and more frequently. By introducing a type of
gamma process (SEM-Gamma) we dampen the Volterra kernel in (\ref{eq:SEM})
by an exponential function and make the process oscillate around a
mean value obtaining a more stable intermittency effect in the Hurst
function.
\begin{defn}
We say that a function $f:\left[0,T\right]\times\mathbb{R}\rightarrow\mathbb{R}_{+}$
is a dampening function if it is nonnegative, satisfies the following
Lipschitz conditions for all $x,y\in\mathbb{R}$ and $t,t^{\prime}\in\left[0,T\right]$
\[
\left|f\left(t,x\right)-f\left(t,y\right)\right|\leq C\left|x-y\right|,
\]
\[
\left|f\left(t,x\right)-f\left(t^{\prime},x\right)\right|\leq C\left|t-t^{\prime}\right|,
\]
and satisfies the following linear growth condition for all $x\in\mathbb{R}$
and $t\in\left[0,T\right]$
\[
\left|f\left(t,x\right)\right|\leq C\left(1+\left|x\right|\right),
\]
 for some constant $C>0$. 
\end{defn}

Let $f$ be a dampening function and let $h$ be a Hurst function
with parameters $\left(h_{*},h^{*}\right)$. The self-excited multifractional
gamma process is given formally by the Volterra equation 
\begin{equation}
X_{t}^{h,f}=\int_{0}^{t}\exp\left(-f\left(t,X_{s}^{h,f}\right)\left(t-s\right)\right)\left(t-s\right)^{h(t,X_{s}^{h,f})-\frac{1}{2}}dB_{s}.\label{eq: SEMP_OU-Process}
\end{equation}

The following lemma shows the existence and uniqueness of the above
equation by means of Theorem \ref{thm:Zhang Existence thm}.
\begin{lem}
\label{lem:SEM-OU_Well_defined}Let $\sigma(t,s,x)=\exp\left(-f\left(t,x\right)\left(t-s\right)\right)\left(t-s\right)^{h(t,x)-\frac{1}{2}}$,
such that $f$ is a dampening function and $h$ is a Hurst function
with parameters $\left(h_{*},h^{*}\right)$. Then, we have that 
\begin{equation}
\left|\sigma\left(t,s,x\right)\right|^{2}\leq k\left(t,s\right)\left(1+\left|x\right|^{2}\right),\label{eq:sigma linear growth-gamma}
\end{equation}
 where 
\[
k\left(t,s\right)=C_{T}\left(t-s\right)^{2h_{*}-1},
\]
 and 
\begin{equation}
\left|\sigma\left(t,s,x\right)-\sigma\left(t,s,y\right)\right|^{2}\leq C_{T}k\left(t,s\right)\left|\log\left(t-s\right)\right|^{2}\left|x-y\right|^{2}.\label{eq:Lipschitz sigma-gamma}
\end{equation}
Moreover, $\sigma$ satisfies $\mathbf{H1}$-$\mathbf{H2}$. \\
\end{lem}

Proofs for all the results in this section are reported in the appendix,
since they are analogous to the ones provided previously for the SEM
process. 

Now since the new $\sigma$ proposed for the SEM-Gamma process also
satisfies $\mathbf{H1}$-$\mathbf{H2}$ we can apply again Zhang's
theorem, in the same way we did in Theorem \ref{thm:Existence and unique},
yielding the existence and uniqueness and bounds on $p$-moments for
the solution of 
\begin{equation}
X_{t}^{h,f}=g\left(t\right)+\int_{0}^{t}e^{-f\left(t,X_{t}^{h,f}\right)\left(t-s\right)}\left(t-s\right)^{h\left(t,X_{t}^{h,f}\right)-\frac{1}{2}}dB_{s}.\label{eq: SEM-Gammaequation}
\end{equation}

We call this process a \textit{Self-Exciting Multifractional Gamma
process (SEM-Gamma)}.

The following Lemma is key to study the Hölder regularity for the
solution of $\left(\ref{eq: SEM-Gammaequation}\right)$, which coincides
and can be derived in the same way as for the SEM process. It is also
useful for the discussion of the strong convergence of the approximating
scheme given in Theorem $\left(\ref{thm:EM_StrongConvergenceGamma}\right).$
\begin{lem}
\label{lem:LambdaFunctionGamma} Let $\sigma(t,s,x)=\exp\left(-f\left(t,x\right)\left(t-s\right)\right)\left(t-s\right)^{h(t,x)-\frac{1}{2}}$,
such that $f$ is a dampening function and $h$ is a Hurst function
with parameters $\left(h_{*},h^{*}\right)$. Then, for any $0<\gamma<2h_{*}$
there exists $\lambda_{\gamma}:\triangle^{(3)}\left(\left[0,T\right]\right)\rightarrow\mathbb{R}$
such that
\begin{equation}
\left|\sigma\left(t,s,x\right)-\sigma\left(t^{\prime},s,x\right)\right|^{2}\leq\lambda_{\gamma}\left(t,t^{\prime},s\right)\left(1+\left|x\right|^{2}\right),\label{eq:time reg sigma-gamma}
\end{equation}
and
\begin{equation}
\int_{0}^{t^{\prime}}\lambda_{\gamma}\left(t,t^{\prime},s\right)ds\leq C_{T,\gamma}\left|t-t^{\prime}\right|^{\gamma},\label{eq:TimeLispchitzIntegralLambda-gamma}
\end{equation}
for some constant $C_{T,\gamma}>0.$
\end{lem}

In order to simulate this process we will use, again, the Euler-Maruyama
approximation to discretize the continuous equation $\left(\ref{eq: SEMP_OU-Process}\right)$.
Consider a time discretization of the interval $\left[0,T\right],$
using a step-size $\Delta t=\frac{T}{N}>0$. The EM method yields
a discrete time approximation $\bar{X}_{k}^{h,f}$ of the process
$X_{t_{k}}^{h,f}$ for $t_{k}=k\Delta t$ with $k\in\left\{ 0,\ldots,N\right\} .$
Therefore we have the following discrete time equation
\begin{align}
\bar{X}_{0}^{h,f} & =X_{0}^{h,f}=0\\
\bar{X}_{k}^{h,f} & =\sum_{i=0}^{k-1}\exp\left(-f\left(t_{k},\bar{X}_{i}^{h,f}\right)\left(t_{k}-t_{i}\right)\right)\left(t_{k}-t_{i}\right)^{h\left(t_{k},\bar{X}_{i}^{h,f}\right)-\frac{1}{2}}\Delta B_{i}\quad\forall k\in\left\{ 1,\ldots,N\right\} ,\label{eq:SEMP-OU-discretization}
\end{align}

where $\Delta B_{i}=B\left(t_{i+1}\right)-B\left(t_{i}\right)$. 

Before trying to implement this approximation, in order to study the
process numerically we will have to prove the following theorem to
ensure the approximation is strongly converging to the process itself.
It will be convenient, just as we did with the SEM process, to consider
a continuous time interpolation of $\left\{ \bar{X}_{k}^{h,f}\right\} _{k\in\left\{ 0,\ldots,N\right\} }$
given by
\begin{equation}
\bar{X}_{t}^{h,f}=\int_{0}^{t}\exp\left(-f\left(t,\bar{X}_{\eta\left(s\right)}^{h,f}\right)\left(t-\eta\left(s\right)\right)\right)\left(t-\eta\left(s\right)\right)^{h\left(t,\bar{X}_{\eta\left(s\right)}^{h,f}\right)-\frac{1}{2}}dB_{s},\quad\forall t\in\left[0,T\right],\label{eq:DefContEulerGamma}
\end{equation}
where, again, $\eta\left(s\right):=t_{i}\cdot\boldsymbol{1}_{\left[t_{i},t_{i+1}\right)}\left(s\right)$.
We also have the following technical result.
\begin{lem}
\label{lem:Bound=000026TimeLipschContEulerGamma} Let $f$ be a dampening
function, $h$ be a Hurst function with parameters $\left(h_{*},h^{*}\right)$
and $\bar{X}^{h,f}=\left\{ \bar{X}_{t}^{h,f}\right\} _{t\in\left[0,T\right]}$
be given by $\left(\ref{eq:DefContEulerGamma}\right)$. Then
\begin{equation}
\mathbb{E}\left[\left|\bar{X}_{t}^{h,f}\right|^{2}\right]\leq C_{T},\quad0\leq t\leq T,\label{eq:SecondMomentContEuler-gamma}
\end{equation}
and 
\begin{equation}
\mathbb{E}\left[\left|\bar{X}_{t}^{h,f}-\bar{X}_{t^{\prime}}^{h,f}\right|^{2}\right]\leq C_{T,\gamma}\left|t-t^{\prime}\right|^{\gamma},\quad0\leq t^{\prime}\leq t\leq T,\label{eq:SecondMomentTLCE-gamma}
\end{equation}
 for any $\gamma<2h_{*}$, where $C_{T}$ and $C_{T,\gamma}$ are
positive constants.
\end{lem}

Using Lemma \ref{lem:Bound=000026TimeLipschContEulerGamma} and Theorem
\ref{theo:VolterraGronwall} we can show the order of convergence
for the approximating scheme.
\begin{thm}
\label{thm:EM_StrongConvergenceGamma}Let $f$ be a dampening function,
$h$ be a Hurst function with parameters $\left(h_{*},h^{*}\right)$
and $\bar{X}^{h,f}=\left\{ \bar{X}_{t}^{h,f}\right\} _{t\in\left[0,T\right]}$
be given by $\left(\ref{eq:DefContEulerGamma}\right)$. Then the Euler-Maruyama
scheme $\left(\ref{eq:DefContEulerGamma}\right)$, satisfies
\begin{equation}
\sup_{0\leq t\leq T}\mathbb{E}\left[\left|X_{t}^{h,f}-\bar{X}_{t}^{h,f}\right|^{2}\right]\leq C_{T,\gamma,h_{*}}E_{h_{*}}\left(C_{T,\gamma,h_{*}}\Gamma\left(h_{*}\right)T^{h_{*}}\right)\left|\Delta t\right|^{\gamma},\label{eq:SEMGamma_StrongConv}
\end{equation}
where $\gamma\in\left(0,2h_{*}\right)$ and $C_{T,\gamma,h_{*}}$
is a positive constant, which does not depend on $N$.
\end{thm}

\begin{example}
We will continue the previous example (\ref{exa:3}). In \cite{FiSo11},
D. Sornette and V. Filimonov suggested a class of self-excited processes
that may exhibit all stylized facts found in financial time series
as heavy tails (asset return distribution displays heavy tails with
positive excess kurtosis), absence of autocorrelations (autocorrelations
in asset returns are negligible, except for very short time scales
$\simeq20$ minutes), volatility clustering (absolute returns display
a positive, significant and slowly decaying autocorrelation function)
and the leverage effect (volatility measures of an asset are negatively
correlated with its returns) among others stated in \cite{CoTa04}.
As we will see, the SEM-Gamma process resembles this properties for
some choices of $h$. The SEM-Gamma process could also be interesting
for modeling commodity markets given that it mean reversion property,
clustering in its increments and also stationary increments, given
by the dampening through the exponential function. The right plot
in Figure $\left(\ref{plot: SEM-Gamma}\right)$, corresponds to a
simulation of a sample path of the process (\ref{eq:SEMP-OU-discretization}),
given the Hurst function $h$ is the same as in example (\ref{exa:3})
given by $h\left(x\right)=\frac{1}{1+x^{2}}$. Notice also we have
taken in this first example of the SEM-Gamma process $f\left(x\right)=0,$
which provides the regular SEM process of the previous section just
to show the left plot looks very similar to the left plot in Figure
$\left(\ref{plot: SEM_ProcesEX3}\right)$.

\begin{figure}
\centering{}\includegraphics[width=0.49\textwidth]{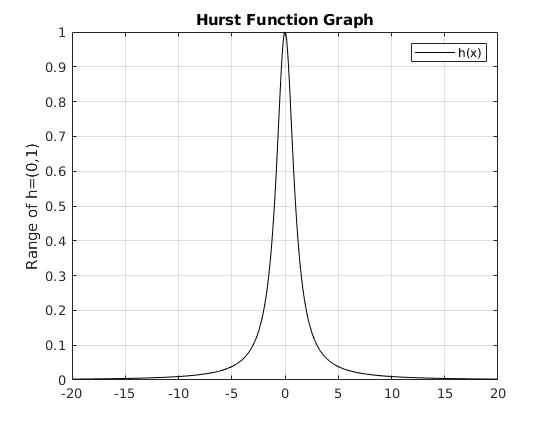}\includegraphics[bb=0bp 0bp 560bp 420bp,width=0.49\textwidth]{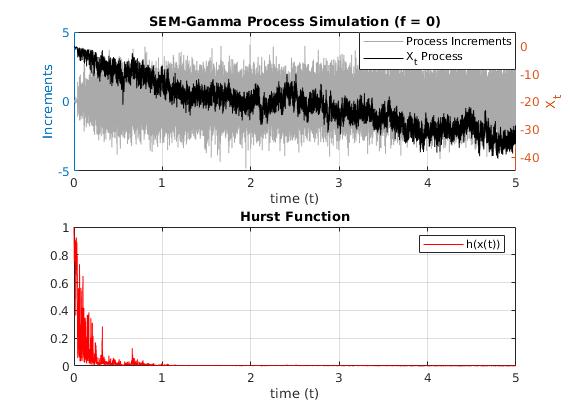}\caption[SEM-Gamma Process Simulation]{Numerical simulation of a trajectory of a SEM-Gamma Process given
the Hurst function is $f=0$ and $h(x)=\frac{1}{1+x^{2}}$.}
\label{plot: SEM-Gamma}
\end{figure}
\label{exa:4_comparative} Figure $\left(\ref{plot: SEM-Gamma_Comparative}\right)$
shows the change in the behavior of the Hurst exponent (a transition
from rougher values to smoother values, i.e. $h\approx0$ to $h\approx1$)
as we shift from lower values for speed of mean reversion, i.e. $f$,
to higher values. In particular we compare $f\in\left\{ 0,0.5,1,10\right\} $

\begin{figure}
\begin{centering}
\includegraphics[width=0.49\textwidth]{EX4_SEMGammaProcess_a=0}\includegraphics[width=0.49\textwidth]{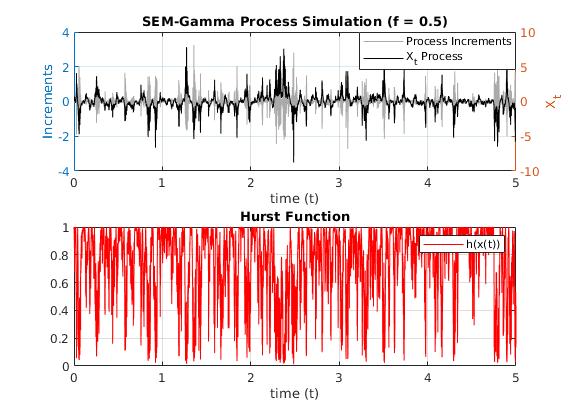}
\par\end{centering}
\begin{centering}
\includegraphics[width=0.49\textwidth]{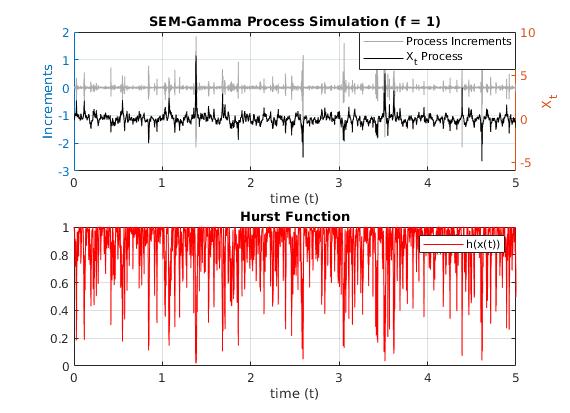}\includegraphics[width=0.49\textwidth]{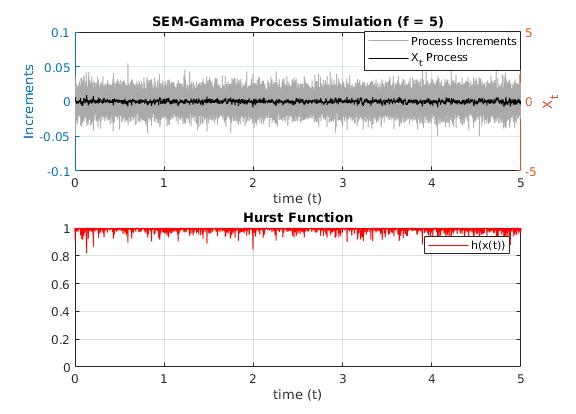}
\par\end{centering}
\caption[SEM-Gamma Process with different speeds of mean reversion]{Numerical simulations of trajectories of SEM-Gamma processes given
$h\left(x\right)=\frac{1}{1+x^{2}}$ and $f\in\left\{ 0,0.5,1,10\right\} $.}
\label{plot: SEM-Gamma_Comparative}
\end{figure}
\end{example}

\begin{rem}
Notice one can control the clustering effect of the increments and
the varying regularity of the process by controlling the parameter
$f$, regardless of the Hurst function chosen as $h$. This is desirable
in numerous fields, for example in financial markets modeling, when
trying to capture shocks in asset prices. It is also important to
remark the fact that using $f\left(x\right)=5,$ we reduced the amount
of spikes to none, shifting the process nature from very rough and
big drift, to a very smooth and driftless process. The following Figure
$\left(\ref{plot: scale comparison}\right)$ shows how by zooming
in the case $f\left(x\right)=5$ close enough we observe the rough
nature hidden at a lower scale.

\begin{figure}
\begin{centering}
\includegraphics[width=0.49\textwidth]{EX4_SEMGammaProcess_a=5_smallview}\includegraphics[width=0.49\textwidth]{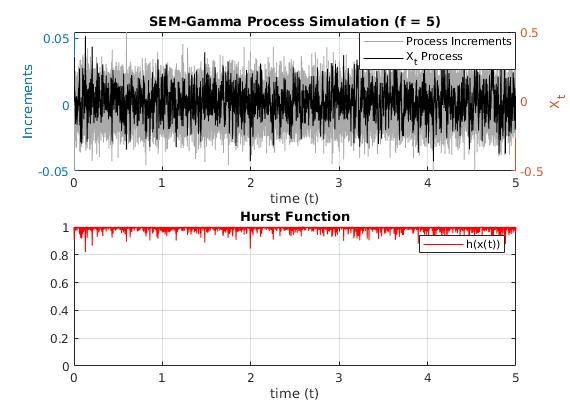}
\par\end{centering}
\caption[SEM-Gamma Process hidden roughness]{Scale comparative of the SEM-Gamma process with $f\left(x\right)=5$
and $h\left(x\right)=\frac{1}{1+x^{2}}$.}
\label{plot: scale comparison}
\end{figure}
It also makes sense to let $f\left(x\right)$ be a function of $x$,
rather than a constant and in particular, if we take $f\left(x\right)=h\left(x\right)=\frac{1}{1+x^{2}}$,
we can see in the following Figure $\left(\ref{plot: SEM-Gamma h1equalh2}\right)$
how the regime switch in the Hurst exponent is less abrupt favoring
sustained difference of roughness in time.

\begin{figure}
\begin{centering}
\includegraphics[width=0.6\textwidth]{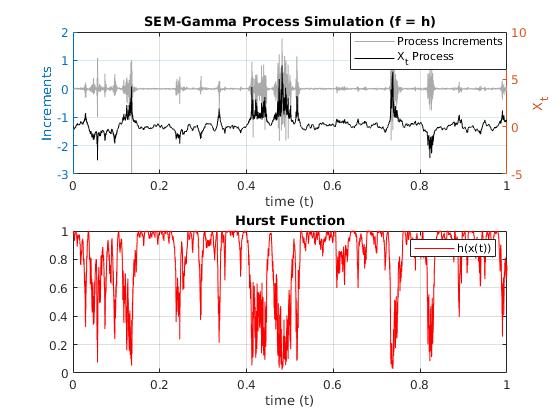}
\par\end{centering}
\caption[SEM-Gamma Process Simulation]{Numerical simulation of a trajectory of a SEM-Gamma Process given
the Hurst function is $f\left(x\right)=h(x)=\frac{1}{1+x^{2}}$}
\label{plot: SEM-Gamma h1equalh2}
\end{figure}
\end{rem}

\newpage{}
\begin{rem}
The plots in Figure $\left(\ref{plot: ACF}\right)$ show the autocorrelation
function of the absolute value in the time series of the increments
in the SEM process (left graph) from example (\ref{exa:3}) and in
the SEM-Gamma process (right graph) with $f\left(x\right)=0.1$. Notice
that autocorrelation in the second case is clearly much higher. 

\begin{figure}
\centering{}\includegraphics[width=0.49\textwidth]{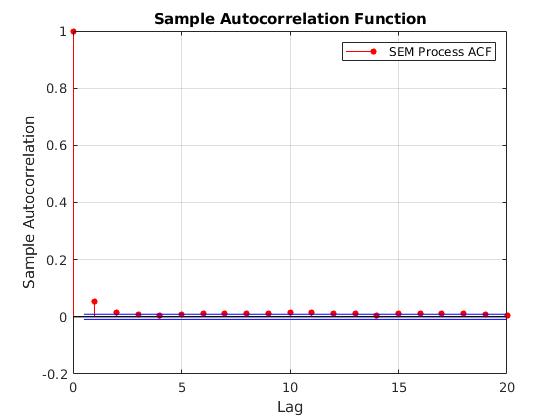}\includegraphics[width=0.49\textwidth]{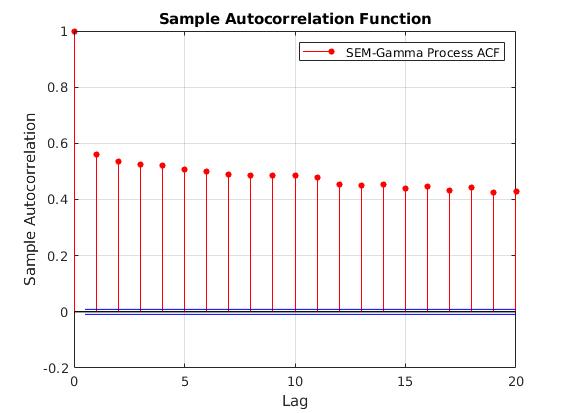}\caption[ACF Functions]{SEM and SEM-Gamma Processes Autocorrelation Function.}
\label{plot: ACF}
\end{figure}
\end{rem}

\newpage{}

\section{Appendix}

In this appendix we have placed the proofs for the results related
with SEM-Gamma process since they are analogous to the proofs in previous
sections.

\subsection{Proof of Lemma \ref{lem:SEM-OU_Well_defined}}
\begin{proof}
We will again proof the three results stated in the lemma by reducing
them to the case proved in Lemma \ref{lem:Well defined sigma}. To
do so we start by proving equation $\left(\ref{eq:sigma linear growth-gamma}\right)$.
By definition, we have that
\begin{align*}
\sigma\left(t,s,x\right) & =\exp\left(-f\left(t,x\right)\left(t-s\right)\right)\left(t-s\right)^{h\left(t,x\right)-\frac{1}{2}}.
\end{align*}
Note that 
\[
\exp\left(-f\left(t,x\right)\left(t-s\right)\right)\leq1,
\]
 since $f\geq0$ and $s<t$, for all $\left(t,s\right)\in\triangle^{\left(2\right)}\left(\left[0,T\right]\right)$.
We also have that $h\left(t,x\right)\in\left[h_{*},h^{*}\right]\subset\left(0,1\right)$
for all $\left(t,x\right)\in\left[0,T\right]\times\mathbb{R}$. Therefore
the result trivially follows from Lemma \ref{lem:Well defined sigma}.

Next we consider equation $\left(\ref{eq:Lipschitz sigma-gamma}\right)$,
using the fact that we can rewrite $\sigma\left(t,s,x\right)$ as
\[
\sigma\left(t,s,x\right)=\exp\left(-f\left(t,x\right)\left(t-s\right)+\log\left(t-s\right)\left(h\left(t,x\right)-\frac{1}{2}\right)\right),
\]
and make use again of inequality $\left(\ref{eq: expinequality}\right)$,
for all $x,y\in\mathbb{R}.$ Now we can write the following upper
bound
\begin{align*}
 & \left|\sigma\left(t,s,x\right)-\sigma\left(t,s,y\right)\right|\\
 & \leq\exp\left(\max\left(-f\left(t,x\right),-f\left(t,y\right)\right)\left(t-s\right)\right)\exp\left(\left(\max\left(h\left(t,x\right),h\left(t,y\right)\right)-\frac{1}{2}\right)\cdot\log\left(t-s\right)\right)\\
 & \quad\times\left(\left|f\left(t,x\right)-f\left(t,y\right)\right|\left|t-s\right|+\left|\log\left(t-s\right)\right|\left|h\left(t,x\right)-h\left(t,y\right)\right|\right).
\end{align*}
Recalling that $\left|e^{-f\left(t,x\right)\left(t-s\right)}\right|\leq1$
and that $f$ and $h$ are uniformly Lipschitz, we have that
\begin{align*}
 & \left|\sigma\left(t,s,x\right)-\sigma\left(t,s,y\right)\right|^{2}\\
 & \leq C^{2}\exp\left(\max\left(\log\left(t-s\right)\left(2h\left(t,x\right)-1\right),\log\left(t-s\right)\left(2h\left(t,y\right)-1\right)\right)\right)\\
 & \quad\times\left|\log\left(t-s\right)\right|^{2}\left|x-y\right|^{2}.
\end{align*}
This reduces the proof to the previous case of a SEM process, see
Lemma \ref{lem:Well defined sigma}. 
\end{proof}

\subsection{Proof of Lemma \ref{lem:LambdaFunctionGamma}}
\begin{proof}
In order to prove equation $\left(\ref{eq:time reg sigma-gamma}\right)$,
it is clear that
\begin{align*}
\sigma\left(t,s,x\right)-\sigma\left(t^{\prime},s,x\right) & =e^{-f\left(t,x\right)\left(t-s\right)}\left(t-s\right)^{h\left(t,x\right)-\frac{1}{2}}-e^{-f\left(t^{\prime},x\right)\left(t^{\prime}-s\right)}\left(t^{\prime}-s\right)^{h\left(t^{\prime},x\right)-\frac{1}{2}}.
\end{align*}
Furthermore, notice that for all $t>t^{\prime}>s>0,$ we can add and
subtract the term 
\[
e^{-f\left(t,x\right)\left(t-s\right)}\left(t-s\right)^{h(t^{\prime},x)-\frac{1}{2}},
\]
 to get
\[
\sigma\left(t,s,x\right)-\sigma\left(t^{\prime},s,x\right)=\tilde{J}^{1}\left(t,t^{\prime},s,x\right)+\tilde{J}^{2}\left(t,t^{\prime},s,x\right),
\]
where
\begin{align*}
\tilde{J}^{1}\left(t,t^{\prime},s,x\right) & :=e^{-f\left(t,x\right)\left(t-s\right)}\left(\left(t-s\right)^{h\left(t,x\right)-\frac{1}{2}}-\left(t-s\right)^{h(t^{\prime},x)-\frac{1}{2}}\right),\\
\tilde{J}^{2}\left(t,t^{\prime},s,x\right) & :=e^{-f\left(t,x\right)\left(t-s\right)}\left(t-s\right)^{h(t^{\prime},x)-\frac{1}{2}}-e^{-f\left(t^{\prime},x\right)\left(t^{\prime}-s\right)}\left(t^{\prime}-s\right)^{h\left(t^{\prime},x\right)-\frac{1}{2}}.
\end{align*}
First we bound $\tilde{J}^{1}$ by using that $e^{-f\left(t,x\right)\left(t-s\right)}\leq1,$
\begin{align*}
\left|\tilde{J}^{1}\left(t,t^{\prime},s,x\right)\right| & \leq\left|J^{1}\left(t,t^{\prime},s,x\right)\right|\leq C_{T,\delta}\left|t-t^{\prime}\right|\left|t'-s\right|^{h_{*}-\frac{1}{2}-\delta},
\end{align*}
since $s<t^{\prime}<t,$ and where $J^{1}$ is the terms appearing
in the proof of Lemma \ref{lem:LambdaFunction}. Next, in order to
bound the term $\tilde{J}^{2}$, we add and subtract the quantity
\[
e^{-f\left(t,x\right)\left(t-s\right)}\left(t^{\prime}-s\right)^{h\left(t^{\prime},x\right)-\frac{1}{2}},
\]
 to obtain
\begin{align*}
\left|\tilde{J}^{2}\left(t,t^{\prime},s,x\right)\right| & \leq\left|e^{-f\left(t,x\right)\left(t-s\right)}\left(\left(t-s\right)^{h(t^{\prime},x)-\frac{1}{2}}-\left(t^{\prime}-s\right)^{h\left(t^{\prime},x\right)-\frac{1}{2}}\right)\right.\\
 & \qquad\left.-\left(t^{\prime}-s\right)^{h\left(t^{\prime},x\right)-\frac{1}{2}}\left(e^{-f\left(t,x\right)\left(t-s\right)}-e^{-f\left(t^{\prime},x\right)\left(t^{\prime}-s\right)}\right)\right|\\
 & \leq\left|e^{-f\left(t,x\right)\left(t-s\right)}\right|\left|\left(t-s\right)^{h(t^{\prime},x)-\frac{1}{2}}-\left(t^{\prime}-s\right)^{h\left(t^{\prime},x\right)-\frac{1}{2}}\right|\\
 & \qquad+\left|\left(t^{\prime}-s\right)^{h\left(t^{\prime},x\right)-\frac{1}{2}}\right|\left|e^{-f\left(t,x\right)\left(t-s\right)}-e^{-f\left(t^{\prime},x\right)\left(t^{\prime}-s\right)}\right|\\
 & \leq\left|J^{2}\left(t,t^{\prime},s,x\right)\right|+\left|\left(t^{\prime}-s\right)^{h\left(t^{\prime},x\right)-\frac{1}{2}}\right|\left|e^{-f\left(t,x\right)\left(t-s\right)}-e^{-f\left(t^{\prime},x\right)\left(t^{\prime}-s\right)}\right|,
\end{align*}
where $J^{2}$ is the term appearing in the proof of Lemma \ref{lem:LambdaFunction}.
Using inequality $\left(\ref{eq: expinequality}\right)$ we can rewrite
the previous expression as
\begin{align*}
\left|\tilde{J}^{2}\left(t,t^{\prime},s,x\right)\right| & \leq\left|J^{2}\left(t,t^{\prime},s,x\right)\right|+\left|\left(t^{\prime}-s\right)^{h\left(t^{\prime},x\right)-\frac{1}{2}}\right|\\
 & \qquad\times\left|e^{\max\left(-f\left(t,x\right)\left(t-s\right),-f\left(t^{\prime},x\right)\left(t^{\prime}-s\right)\right)}\right|\left|f\left(t^{\prime},x\right)\left(t^{\prime}-s\right)-f\left(t,x\right)\left(t-s\right)\right|\\
 & \leq\left|J^{2}\left(t,t^{\prime},s,x\right)\right|+C_{T}\left|t^{\prime}-s\right|^{h\left(t^{\prime},x\right)-\frac{1}{2}}\left|t-t^{\prime}\right|\left|f\left(t^{\prime},x\right)-f\left(t,x\right)\right|.
\end{align*}
Then, adding and subtracting $f\left(t,x\right)\left(t'-s\right)$,
and using the linear growth and Lipschitz conditions on $f$, we obtain
\begin{align*}
\left|f\left(t^{\prime},x\right)\left(t^{\prime}-s\right)-f\left(t,x\right)\left(t-s\right)\right| & \leq\left|f\left(t,x\right)\right|\left|t'-t\right|+\left|t'-s\right|\left|f\left(t',x\right)-f\left(t,x\right)\right|\\
 & \leq C\left|t'-t\right|\left(1+\left|x\right|\right)+C\left|t'-s\right|\left|t'-t\right|,
\end{align*}
and we can conclude that
\begin{align*}
\left|\tilde{J}^{2}\left(t,t^{\prime},s,x\right)\right| & \leq\left|J^{2}\left(t,t^{\prime},s,x\right)\right|+C\left|t^{\prime}-s\right|^{h\left(t^{\prime},x\right)-\frac{1}{2}}\left|t-t^{\prime}\right|\left(1+\left|x\right|\right)\\
 & \leq\left|J^{2}\left(t,t^{\prime},s,x\right)\right|+C_{T}\left|t^{\prime}-s\right|^{h_{*}-\frac{1}{2}}\left|t-t^{\prime}\right|\left(1+\left|x\right|\right).
\end{align*}
Therefore, if we define 
\[
\lambda_{\gamma}\left(t,t^{\prime},s\right):=C_{T,\gamma}\left(t-t^{\prime}\right)^{\gamma}\left(t^{\prime}-s\right)^{-1+h_{*}-\frac{\gamma}{2}},
\]
for $0<\gamma<2h_{*},$ and use the final bounds for $J^{1}$ and
$J^{2}$ in Lemma \ref{lem:LambdaFunction}, we get that
\begin{align*}
 & \left|\sigma\left(t,s,x\right)-\sigma\left(t^{\prime},s,x\right)\right|^{2}\\
 & \leq4\left(\left|J^{1}\left(t,t^{\prime},s,x\right)\right|^{2}+\left|J^{2}\left(t,t^{\prime},s,x\right)\right|^{2}+\left|C_{T}\left|t^{\prime}-s\right|^{h_{*}-\frac{1}{2}}\left|t-t^{\prime}\right|\left(1+\left|x\right|\right)\right|^{2}\right)\\
 & \leq\lambda_{\gamma}\left(t,t^{\prime},s\right)\left(1+\left|x\right|^{2}\right),
\end{align*}
and 
\[
\int_{0}^{t^{\prime}}\lambda_{\gamma}\left(t,t^{\prime},s\right)ds\leq C_{T,\gamma}\left(t^{\prime}\right)^{h_{*}-\frac{\gamma}{2}}\left(t-t^{\prime}\right)^{\gamma},
\]
 which concludes the proof.
\end{proof}

\subsection{Proof of Lemma \ref{lem:Bound=000026TimeLipschContEulerGamma}}
\begin{proof}
Recall that $k\left(t,s\right)=C_{T}\left(t-s\right)^{2h_{*}-1}$
and, since $\eta\left(s\right)\leq s$, we have the following inequality

\begin{equation}
k\left(t,\eta\left(s\right)\right)\leq k\left(t,s\right).\label{eq: Volterrabound-1}
\end{equation}
Using the Itô isometry, that $e^{-2f\left(t,x\right)}\leq1$, equation
$\left(\ref{eq:BoundSigma2_No_x}\right)$ and equation (\ref{eq: Volterrabound-1}),
we obtain
\begin{align*}
\mathbb{E}\left[\left|\bar{X}_{t}^{h,f}\right|^{2}\right] & =\mathbb{E}\left[\int_{0}^{t}e^{-2f\left(t,\bar{X}_{\eta\left(s\right)}^{h,f}\right)\left(t-\eta\left(s\right)\right)}\left(t-\eta\left(s\right)\right)^{2h\left(t,\bar{X}_{\eta\left(s\right)}^{h,f}\right)-1}ds\right]\\
 & \leq\mathbb{E}\left[\int_{0}^{t}\left(t-\eta\left(s\right)\right)^{2h\left(t,\bar{X}_{\eta\left(s\right)}^{h,f}\right)-1}ds\right]\\
 & \leq\int_{0}^{t}k\left(t,\eta\left(s\right)\right)ds\leq\int_{0}^{t}k\left(t,s\right)ds\leq C_{T}.
\end{align*}
To prove the bound $\left(\ref{eq:SecondMomentTLCE}\right),$ note
that
\begin{align*}
\bar{X}_{t}^{h,f}-\bar{X}_{t^{\prime}}^{h,f} & =\int_{t^{\prime}}^{t}e^{-f\left(t,\bar{X}_{\eta\left(s\right)}^{h,f}\right)\left(t-\eta\left(s\right)\right)}\left(t-\eta\left(s\right)\right)^{h\left(t,\bar{X}_{\eta\left(s\right)}^{h,f}\right)-\frac{1}{2}}dB_{s},\\
 & \quad+\int_{0}^{t^{\prime}}\left\{ e^{-f\left(t,\bar{X}_{\eta\left(s\right)}^{h,f}\right)\left(t-\eta\left(s\right)\right)}\left(t-\eta\left(s\right)\right)^{h\left(t,\bar{X}_{\eta\left(s\right)}^{h,f}\right)-\frac{1}{2}}\right.\\
 & \qquad\qquad\left.-e^{-f\left(t',\bar{X}_{\eta\left(s\right)}^{h,f}\right)\left(t-\eta\left(s\right)\right)}\left(t^{\prime}-\eta\left(s\right)\right)^{h\left(t^{\prime},\bar{X}_{\eta\left(s\right)}^{h,f}\right)-\frac{1}{2}}\right\} dB_{s}\\
 & =:J_{1}+J_{2}.
\end{align*}
 Due to the Itô isometry, that $e^{-2f\left(t,x\right)}\leq1$, equation
$\left(\ref{eq:BoundSigma2_No_x}\right)$ and $\left(\ref{eq: Volterrabound-1}\right)$,
we obtain the bounds

\begin{align*}
\mathbb{E}\left[\left|J_{1}\right|^{2}\right] & =\mathbb{E}\left[\int_{t^{\prime}}^{t}e^{-2f\left(t,\bar{X}_{\eta\left(s\right)}^{h,f}\right)\left(t-\eta\left(s\right)\right)}\left(t-\eta\left(s\right)\right)^{2h\left(t,\bar{X}_{\eta\left(s\right)}^{h,f}\right)-1}ds\right]\\
 & \leq\int_{t'}^{t}k\left(t,\eta\left(s\right)\right)ds\leq\int_{t'}^{t}k\left(t,s\right)ds=C_{T}\left|t-t^{\prime}\right|^{2h_{*}}.
\end{align*}
Using again the Itô isometry, equation $\left(\ref{eq:time reg sigma-gamma}\right)$
and equation $\left(\ref{eq:TimeLispchitzIntegralLambda-gamma}\right)$
we can write, for any $\gamma<2h_{*},$ that
\begin{align*}
\mathbb{E}\left[\left|J_{2}\right|^{2}\right] & \leq\int_{0}^{t^{\prime}}\lambda_{\gamma}\left(t,t^{\prime},\eta\left(s\right)\right)\left(1+\mathbb{E}\left[\left|\bar{X}_{\eta\left(s\right)}^{h,f}\right|^{2}\right]\right)ds\leq C_{T}\int_{0}^{t^{\prime}}\lambda_{\gamma}\left(t,t^{\prime},s\right)ds\leq C_{T,\gamma}\left|t-t^{\prime}\right|^{\gamma},
\end{align*}
where in the second inequality we have used that $\lambda_{\gamma}\left(t,t^{\prime},\eta\left(s\right)\right)\leq\lambda_{\gamma}\left(t,t^{\prime},s\right)$,
because $\lambda_{\gamma}$ is essentially a negative fractional power
of $(t-s)$ and $\eta\left(s\right)\leq s$ and also that $\mathbb{E}\left[\left|\bar{X}_{t}^{h,f}\right|^{2}\right]\leq C_{T}$,
$0\leq t\leq T$, which we just have proved above. Combining the bounds
for $\mathbb{E}\left[\left|J_{1}\right|^{2}\right]$ and $\mathbb{E}\left[\left|J_{2}\right|^{2}\right]$
the result follows.
\end{proof}

\subsection{Proof of Theorem \ref{thm:EM_StrongConvergenceGamma}}
\begin{proof}
We will reduce the proof to the case in Theorem \ref{thm:EM_StrongConvergence}.
To do so, in the same way we did, we define
\[
\delta_{t}:=X_{t}^{h,f}-\bar{X}_{t}^{h,f},\qquad\varphi\left(t\right):=\sup_{0\leq s\leq t}\mathbb{E}\left[\left|\delta_{s}\right|^{2}\right],\quad t\in\left[0,T\right].
\]
For any $t\in\left[0,T\right]$, we can write
\begin{align*}
\delta_{t} & =\int_{0}^{t}\left(e^{-f\left(t,X_{s}^{h,f}\right)\left(t-s\right)}\left(t-s\right)^{h(t,X_{s}^{h,f})-\frac{1}{2}}\right.\\
 & \qquad-\left.e^{-f\left(t,\bar{X}_{\eta\left(s\right)}^{h,f}\right)\left(t-\eta\left(s\right)\right)}\left(t-\eta\left(s\right)\right)^{h\left(t,\bar{X}_{\eta\left(s\right)}^{h,f}\right)-\frac{1}{2}}\right)dB_{s}\\
 & =\int_{0}^{t}\left(e^{-f\left(t,X_{s}^{h,f}\right)\left(t-s\right)}\left(t-s\right)^{h(t,X_{s}^{h,f})-\frac{1}{2}}\right.\\
 & \qquad-\left.e^{-f\left(t,X_{s}^{h,f}\right)\left(t-s\right)}\left(t-s\right)^{h\left(t,\bar{X}_{\eta\left(s\right)}^{h,f}\right)-\frac{1}{2}}\right)dB_{s}\\
 & +\int_{0}^{t}\left(e^{-f\left(t,X_{s}^{h,f}\right)\left(t-s\right)}\left(t-s\right)^{h\left(t,\bar{X}_{\eta\left(s\right)}^{h,f}\right)-\frac{1}{2}}\right.\\
 & \qquad-\left.e^{-f\left(t,\bar{X}_{\eta\left(s\right)}^{h,f}\right)\left(t-\eta\left(s\right)\right)}\left(t-\eta\left(s\right)\right)^{h\left(t,\bar{X}_{\eta\left(s\right)}^{h,f}\right)-\frac{1}{2}}\right)dB_{s}\\
 & =:\tilde{I}_{1}\left(t\right)+\tilde{I}_{2}\left(t\right).
\end{align*}
First we bound the second moment of $\tilde{I}_{1}\left(t\right)$
in terms of a certain integral of $\varphi$. Using the Itô isometry,
equation $\left(\ref{eq:Lipschitz sigma-gamma}\right)$ and the Lipschitz
property of $h$ we get
\begin{align*}
\mathbb{E}\left[\left|\tilde{I}_{1}\left(t\right)\right|^{2}\right] & \leq\int_{0}^{t}k\left(t,s\right)\left(\log\left(t-s\right)\right)^{2}\mathbb{E}\left[\left(h\left(t,X_{s}^{h,f}\right)-h\left(t,\bar{X}_{\eta\left(s\right)}^{h,f}\right)\right)^{2}\right]ds\\
 & \leq C_{T,\delta}\int_{0}^{t}\left(t-s\right)^{2\left(h_{*}-\delta\right)-1}\mathbb{E}\left[\left|X_{s}^{h,f}-\bar{X}_{\eta\left(s\right)}^{h,f}\right|^{2}\right]ds,
\end{align*}
for $\delta>0,$ arbitrarily small. By the same arguments as in the
proof of Theorem \ref{thm:EM_StrongConvergence} we obtain the following
bound
\begin{equation}
\mathbb{E}\left[\left|\tilde{I}_{1}\right|^{2}\right]\leq C_{T,h_{*}}\left\{ \int_{0}^{t}\left(t-s\right)^{h_{*}-1}\varphi\left(s\right)ds+\left|\Delta t\right|^{\gamma}\right\} .\label{eq: barI1Bound}
\end{equation}
Next, we find a bound for the second moment of $\tilde{I}_{2}\left(t\right)$.
Using again the Itô isometry, equations $\left(\ref{eq:time reg sigma-gamma}\right)$
and $\left(\ref{eq:TimeLispchitzIntegralLambda-gamma}\right)$, and
Lemma \ref{lem:Bound=000026TimeLipschContEuler} we can write
\begin{equation}
\mathbb{E}\left[\left|\tilde{I}_{2}\right|^{2}\right]\leq\int_{0}^{t}\lambda_{\gamma}\left(t+\left(s-\eta\left(s\right)\right),t,s\right)\left(1+\mathbb{E}\left[\left|\bar{X}_{\eta\left(s\right)}^{h,f}\right|^{2}\right]\right)ds\leq C_{T,\gamma}\left|\Delta t\right|^{\gamma},\label{eq: barI2Bound}
\end{equation}
for any $\gamma<2h_{*}$, and where we have used that 
\[
\mathbb{E}\left[\left|\bar{X}_{s}^{h,f}\right|^{2}\right]\leq C_{T},\qquad0\leq s\leq T.
\]
Combining the inequalities $\left(\ref{eq: barI1Bound}\right)$ and
$\left(\ref{eq: barI2Bound}\right)$ we obtain
\[
\tilde{\varphi}\left(t\right)\leq C_{T,\gamma,h_{*}}\left\{ \int_{0}^{t}\left(t-s\right)^{h_{*}-1}\varphi\left(s\right)ds+\left|\Delta t\right|^{\gamma}\right\} .
\]
Using again Lemma \ref{theo:VolterraGronwall} we can conclude.
\end{proof}
\bibliographystyle{APT}
\bibliography{SEMP}

\end{document}